\documentclass[a4paper, 12pt]{amsart}
\usepackage{amssymb}
\usepackage{amsthm}
\usepackage{amsmath}
\usepackage{tikz-cd}
\usepackage{enumitem}
\usepackage{graphicx}
\usepackage[utf8]{inputenc}
\usepackage[english]{babel}

\DeclareMathOperator{\HS}{\mathrm{HS}}

\DeclareMathOperator{\rank}{\mathrm{rank}}

\newtheorem{theorem}{Theorem}[section]
\newtheorem{corollary}[theorem]{Corollary}
\newtheorem{lemma}[theorem]{Lemma}

\newtheorem{proposition}[theorem]{Proposition}

\theoremstyle{definition}
\newtheorem{definition}[theorem]{Definition}
\newtheorem{example}[theorem]{Example}
\newtheorem{remark}[theorem]{Remark}

\title{Preserving Lefschetz properties after extension of variables}

\author{Filip Jonsson Kling}

\address{Department of Mathematics\\ Stockholm University\\ SE-106 91 Stockholm, Sweden}
\email{filip.jonsson.kling@math.su.se}

\begin{document}

\begin{abstract}
Consider a standard graded artinian $k$-algebra $B$ and an extension of $B$ by a new variable, $A=B\otimes_k k[x]/(x^d)$ for some $d\geq 1$. We will show how maximal rank properties for powers of a general linear form on $A$ can be determined by maximal rank properties for different powers of general linear forms on $B$. This is then used to study Lefschetz properties of algebras that can be obtained via such extensions. In particular, it allows for a new proof that monomial complete intersections have the strong Lefschetz property over a field of characteristic zero. Moreover, it gives a recursive formula for the determinants that show up in that case. Finally, for algebras over a field of characteristic zero, we give a classification for what properties $B$ must have for all extensions $B\otimes_k k[x]/(x^d)$ to have the weak or the strong Lefschetz property.
\end{abstract}

\maketitle

\section{Introduction}
Throughout this paper, we will consider standard graded artinian $k$-algebras $A$ for some field $k$. Then $A$ can be written as a quotient of the polynomial ring $k[x_1,\dots, x_n]$ modulo some homogeneous ideal $I$ and thus has a grading $A=A_0\oplus A_1 \oplus \dots \oplus A_D$ where $A_i$ is the vector space consisting of zero and all degree $i$ forms in $A$. The main property of such an algebra that we will be interested in is if $A$ has the so called weak or strong Lefschetz property. An algebra $A$ has the weak Lefschetz property (WLP) if there is a linear form $\ell\in A_1$ such that the multiplication maps
\[
\cdot \ell: A_i \to A_{i+1}
\]
have full rank for any $i$. If all powers of $\ell$ also give maps that have full rank, we say that $A$ has the strong Lefschetz property (SLP). Even though it is usually not too difficult to determine if any given algebra $A$ enjoys the weak or strong Lefschetz property or not, classifying if a family of algebras satisfies those properties or not is a much more difficult task and is a very active area of research within commutative algebra and related fields, see \cite{Height_4, Lefschetz_book, Survey} and the references therein.  

One way to tackle the problem if an algebra has the weak or the strong Lefschetz property is to see if one can break up the algebra into smaller parts. Sometimes one can then use knowledge of how multiplications by general linear forms act on these parts separately to gain insight into how it acts on the whole algebra. For example, if $B$ and $C$ are two algebras with the SLP and symmetric Hilbert series, then $A=B\otimes_k C$ also has the SLP \cite[Theorem 3.34]{Lefschetz_book}. Conversely, if $B$ fails the WLP and the Hilbert series of $C$ has an isolated peak, then $A=B\otimes_k C$ fails the WLP \cite[Theorem 3.1]{Forcing_WLP}. 

The goal of this paper is to classify how the Lefschetz properties behave under such decompositions of an algebra, at least if one of the pieces is simple enough. To be precise, the main result of this paper gives a recursion of the following form.

\begin{theorem}
Let $k$ be a field of characteristic zero, $B$ an artinian $k$-algebra and set $$A=B\otimes_k k[x]/(x^d)$$ for some variable $x$ that does not appear in $B$ and some $d\geq 1$. Pick integers $i,t\geq 0$. If $\overline{\ell}$ is a general linear form in $B$ and $\ell=\overline{\ell} + x$ is a general linear form in $A$, then 
\[
\ell^t:A_i \to A_{i+t}
\]
has full rank if and only if the maps
\[
\cdot \overline{\ell}^{2q + t -(d-1)}:B_{i-q} \to B_{i+q+t-(d-1)}
\]
where $q$ goes from $\max\{0,d-t\}$ to $d-1$ are either all injective, or all surjective. 
\end{theorem}
See Theorem \ref{thm:Rank-recursion} for the full statement, including some positive characteristic and determinant calculations. 
This result thus explains how the Lefschetz properties behaves when adding a new variable to the algebra. In particular, it allows for a new proof that monomial complete intersections have the strong Lefschetz property over fields of characteristic zero, as first proven independently by Stanley \cite{Stanley} and Watanabe \cite{Watanabe}. Moreover, in Theorem~\ref{thm:CI_determinant}, it is used to give a new recursive formula for the determinants arising in the study of the SLP for these monomial complete intersections.

As the proof of Theorem \ref{thm:Rank-recursion} is quite technical, we exhibit in Section \ref{sec:Quadratic} the main ideas of our proof in the quadratic case, when a variable $x$ with $x^2=0$ is added. Section \ref{sec:recursion} is then devoted to establish Theorem \ref{thm:Rank-recursion} of how maximal rank properties for powers of a linear form behave when adding a new variable $x$ with $x^d=0$ for an arbitrary $d\geq 2$. The proof that monomial complete intersections have the SLP then follows in Section \ref{sec:mon_CI}. Finally, some applications of Theorem~\ref{thm:Rank-recursion} are given in Section \ref{sec:Applications}. There we ask what kind of properties an algebra must satisfy in order for families of extensions of it to always satisfy some Lefschetz property. In particular, we give a new proof of a theorem of Lindsey \cite[Theorem 3.10]{Lindsey} and some new variations of it.

\section{The quadratic case}\label{sec:Quadratic}
In this section we want to examine how maximal rank properties behave for artinian algebras when adding a new variable $x$ with $x^2=0$. The arguments in this case are heavily inspired by Phoung and Tran's proof that quadratic monomial complete intersections have the SLP \cite{New_Stanley}. So let 
\[
A=B\otimes_k k[x]/(x^2)
\]
where $B$ is some standard graded artinian $k$-algebra and $x$ is a variable not appearing in $B$. For a general linear form $\overline{\ell}\in B_1$, we then have that a general linear form $\ell\in A_1$ can be written as
\[
\ell = \overline{\ell} + cx
\]
for some nonzero constant $c$ in $k$. However, after a change of variables we may assume that $c=1$, so we will work with
\[
\ell = \overline{\ell} + x.
\]
See \cite[Theorem 2.2]{Sum_of_var_SLP} for details. The question we now would like to answer is if we can determine when the multiplication map
\[
\cdot \ell^t:A_i \to A_{i+t}
\]
has full rank for some given integers $i$ and $t$ given knowledge of when powers of $\overline{\ell}$ give full rank maps on $B$. We will also need a coherence notion for a family of linear maps.

\begin{definition}
Let $f_i:V_i \to W_i$ for $i=1,2,\dots$ be a family of linear maps. We say that they have full rank \emph{for the same reason} if either all $f_i$ are injective or all $f_i$ are surjective.
\end{definition}

Next, for $A=B\otimes_k k[x]/(x^2)$, write $\mathcal{A}_i$ for a basis of $A_i$ and $\mathcal{B}_i$ for a basis of $B_i$. Then we can choose the basis of $A_i$ such that
\begin{equation}\label{eq:A_decomp_quad}
\mathcal{A}_i = \mathcal{B}_i \sqcup x\mathcal{B}_{i-1}.
\end{equation}
We also have that
\begin{equation}\label{eq:ell_decomp_quad}
\ell^t= \overline{\ell}^t + tx\overline{\ell}^{t-1}
\end{equation}
in $A$ for any integer $t\geq 0$ since $x^2=0$. Further, let $M_i^t$ denote the multiplication matrix corresponding to 
\[
\cdot \ell^t:A_i\to A_{i+t}
\]
and $\overline{M}_i^t$ the corresponding matrix for
\[
\cdot \overline{\ell}^t:B_i\to B_{i+t}.
\]
Note that the $t$ in $M_i^t$ does \emph{not} indicate taking the transpose, it only indicates the power of the linear form that the matrix is representing. We now wish to write $M_i^t$ in such a way so that it is easy to determine when it has full rank. The decomposition of $\mathcal{A}_i$ in \eqref{eq:A_decomp_quad} gives that we may assume that $i\geq 0$ and $B_{i+t-1}\neq 0$ as else $A_i=0$ or $A_{i+t}=0$, giving that our map has full rank for trivial reasons. Let us for now assume that $i>0$ and $B_{i+t}\neq 0$ as well to avoid some edge cases. Using the decomposition of $\mathcal{A}_i$ and the expression for $\ell^t$ in \eqref{eq:ell_decomp_quad} then gives that $M_i^t$ decomposes into several blocks as
\[
M_i^t=
\begin{pmatrix}
\overline{M}_i^t & 0 \\
t\overline{M}_i^{t-1} & \overline{M}_{i-1}^t
\end{pmatrix}.
\]
This is a generalization of \cite[Lemma 2.4]{New_Stanley} where such a decomposition is done given that $B$ is quadratic monomial complete intersection. Observe now that since $\overline{\ell}^t = \overline{\ell}^{t-1}\cdot\overline{\ell}= \overline{\ell}\cdot \overline{\ell}^{t-1}$, we can factor the diagonal blocks as
\[
\overline{M}_{i-1}^t=\overline{M}_i^{t-1}\overline{M}_{i-1}^{1} \text{ and } \overline{M}_{i}^t=\overline{M}_{i+t-1}^{1}\overline{M}_{i}^{t-1}.
\]
This shows that $M_i^t$ is of the form such that the next result can be used, which is a more detailed version of \cite[Lemma 3.2]{Quadratic_LEX}.

\begin{lemma}
\label{lem:Block_quadratic}
Let $U,P,D$ be matrices over a field $k$ such that $UPD$ is defined and let $\alpha$ in $k$ be a non-zero element. Then there are square matrices $L, R$ with determinant one such that for 
\[ M=
\begin{pmatrix}
UP & 0 \\
\alpha P & PD 
\end{pmatrix}
\]
we have that
\[
LMR=\begin{pmatrix}
0 & -\alpha^{-1}UPD \\
\alpha P & 0 
\end{pmatrix}.
\]
In particular, $\rank(M)=\rank(P) + \rank(UPD)$ and $M$ has full rank if and only if $P$ and $UPD$ both have full rank for the same reason. Moreover, if $P$ and $UPD$ are square matrices of size $n$ and $m$ respectively, then
\[
\det(M) = (-1)^{m(n+1)}\det(P)\det(UPD)\alpha^{n-m}.
\]
\end{lemma}

\begin{proof}
Choose 
\[
L=\begin{pmatrix}
I & -\alpha^{-1}U \\
0 & I 
\end{pmatrix} \text{ and } R=\begin{pmatrix}
I & -\alpha^{-1}D \\
0 & I 
\end{pmatrix}.
\]
The main statement of the lemma follows from the calculation 
\[
\begin{pmatrix}
I & -\alpha^{-1}U \\
0 & I 
\end{pmatrix}
\begin{pmatrix}
UP & 0 \\
\alpha P & PD 
\end{pmatrix}
\begin{pmatrix}
I & -\alpha^{-1}D \\
0 & I 
\end{pmatrix} = 
\begin{pmatrix}
0 & -\alpha^{-1}UPD \\
\alpha P & 0 
\end{pmatrix}.
\]
The claim about the ranks then follows since $\rank(\alpha E)=\rank(E)$ for any matrix $E$ and an invertible element $\alpha$. Finally, in the case that $P$ and $UPD$ are square of size $n$ and $m$, we get using $\det(L)=\det(R)=1$ that
\begin{align*}
\det 
\begin{pmatrix}
UP & 0 \\
\alpha P & PD 
\end{pmatrix} &= \det 
\begin{pmatrix}
0 & -\alpha^{-1}UPD \\
\alpha P & 0 
\end{pmatrix} \\ 
&= (-1)^{nm}\det 
\begin{pmatrix}
\alpha P & 0 \\
0 & -\alpha^{-1}UPD 
\end{pmatrix} \\
&=(-1)^{nm}\det(\alpha P)\det(-\alpha^{-1}UPD) \\
&=(-1)^{m(n+1)}\det(P)\det(UPD)\alpha^{n-m}
\end{align*}
as desired.
\end{proof}

Applying Lemma \ref{lem:Block_quadratic} to $M=M_i^t$, we have 
\[
P=\overline{M}_i^{t-1}, \quad U=\overline{M}_{i+t-1}^{1}, \quad D=\overline{M}_{i-1}^{1}, \quad \alpha=t
\]
with 
\[
UPD=\overline{M}_{i+t-1}^{1}\overline{M}_i^{t-1}\overline{M}_{i-1}^{1} = \overline{M}_{i-1}^{t+1}
\]
the matrix representing $\cdot \overline{\ell}^{t+1}:\overline{A}_{i-1} \to \overline{A}_{i+t}$. With this preparation, we can now state the main result of this section.

\begin{theorem}
\label{thm:recursion_quadratic}
Let $A=B\otimes_k k[x]/(x^2)$ be a standard graded artinian $k$-algebra and let $\ell=\overline{\ell}+ x$ be a general linear form in $A$. Fix $i,t\geq 1$ and assume that $B_{i+t}\neq 0$ and that $t$ is invertible in $k$. Then $\cdot \ell^t : A_i \to A_{i+t}$ has full rank if and only if the maps
\[
\cdot \overline{\ell}^{t-1}: B_i \to B_{i+t-1}
\text{ and } \cdot \overline{\ell}^{t+1}: B_{i-1} \to B_{i+t}
\]
have full rank for the same reason. Moreover, if the matrices $\overline{M}_i^{t-1}$ and $\overline{M}_{i-1}^{t+1}$ representing these maps are square of size $n$ and $m$ respectively, then the matrix $M_i^t$ representing $\cdot \ell^t$ is square of size $n+m$ and
\[
\det(M_i^t) = (-1)^{m(n+1)}\det(\overline{M}_i^{t-1})\det(\overline{M}_{i-1}^{t+1})t^{n-m}.
\]
\end{theorem}

\begin{proof}
Let $M_i^t$ be the matrix representing $\cdot \ell^t:A_i\to A_{i+t}$. In the context of  Lemma~\ref{lem:Block_quadratic}, we have $P=\overline{M}_i^{t-1}$, $UPD=\overline{M}_{i-1}^{t+1}$ and $\alpha=t$. By our assumption that $t$ is invertible in $k$, we get that $\alpha$ is non-zero. Hence we can apply Lemma~\ref{lem:Block_quadratic} to $M_i^t$ and the results follow.
\end{proof}

\begin{remark}
Note that the formula for the determinant in Theorem \ref{thm:recursion_quadratic} holds in the case when $t$ is not invertible in $k$ as well. Indeed, this is the case since the determinant in characteristic $p$ is just obtained form the formula for the determinant in characteristic zero by reducing mod $p$.
\end{remark}

Let us now examine what happens in the edge cases.

\begin{proposition}
\label{prop:quadratic_edge_case}
Let $A=B\otimes_k k[x]/(x^2)$ be a standard graded artinian $k$-algebra and let $\ell=\overline{\ell}+ x$ be a general linear form in $A$. Fix $i\geq 0$, $t\geq 1$ and assume that $t$ is invertible in $k$. Consider the map $\cdot \ell^t:A_i \to A_{i+t}$ when $i=0$ or when $B_{i+t}=0$. If $i=0$, it has full rank if and only if 
\[
\cdot \overline{\ell}^{t-1}:B_0 \to B_{t-1}
\]
has full rank, while if $i>0$ and $B_{i+t}=0$, it has full rank if and only if 
\[
\cdot \overline{\ell}^{t-1}:B_i \to B_{i+t-1}
\]
is surjective. Moreover, if $i=0$, $B_{t}=0$ and $\dim B_{t-1}=1$, then the matrix $M_0^t$ representing $\cdot \ell^t$ is square with determinant
\[
\det(M_0^t) = t\det(\overline{M}_0^{t-1})
\]
where $\overline{M}_0^{t-1}$ is the matrix representing $\cdot \overline{\ell}^{t-1}:B_0 \to B_{t-1}$.
\end{proposition}

\begin{proof}
If $i=0$, then the decomposition of $\mathcal{A}_0=\mathcal{B}_0$ has only one piece. Hence, if $B_{i+t}\neq0$, the matrix $M_0^t$ representing $\cdot \ell^t:A_0 \to A_{t}$ takes the form
\[
M_0^t = \begin{pmatrix}
\overline{M}_0^t \\
t\overline{M}_0^{t-1}
\end{pmatrix}.
\]
Multiplying by 
\[
L=\begin{pmatrix}
I & -t^{-1}\overline{M}_{t-1}^1 \\
0 & I
\end{pmatrix}
\]
from the left then gives
\[
LM_0^t = \begin{pmatrix}
0 \\
t\overline{M}_0^{t-1}
\end{pmatrix}.
\]
Note that $t^{-1}$ exists by assumption. Hence $LM_0^t$, and therefore also $M_0^t$, has full rank if and only if $\overline{M}_0^{t-1}$ is injective. Since $\dim_k B_0 = 1$, $\overline{M}_0^{t-1}$ is a matrix consisting of a single column, so it is injective if and only if it is non-zero, which is equivalent to $\cdot \overline{\ell}^{t-1}:B_0 \to B_{t-1}$ having full rank.

If $B_{i+t}=0$, then $\mathcal{A}_{i+t}=x\mathcal{B}_{i+t-1}$. Hence $M_i^{t}$ takes the form
\[
M_i^t = \begin{pmatrix}
t\overline{M}_i^{t-1} & \overline{M}_{i-1}^{t}
\end{pmatrix}
\]
in this case. For 
\[
R=\begin{pmatrix}
I & -t^{-1}\overline{M}_{i-1}^1 \\
0 & I
\end{pmatrix}, 
\]
multiplying from the right gives
\[
M_i^{t}R = \begin{pmatrix}
t\overline{M}_i^{t-1} & 0
\end{pmatrix}.
\]
So for $M_i^t$ to have full rank, we require that $\overline{M}_i^{t-1}$, or equivalently $\cdot \overline{\ell}^{t-1}:B_i \to B_{i+t-1}$, is surjective. Finally, if both $i=0$ and $B_{i+t}=0$, then 
\[
M_0^t = t\overline{M}_0^{t-1},
\]
so $M_0^t$ has full rank if and only if $\cdot \overline{\ell}^{t-1}:B_0 \to B_{t-1}$ has full rank and the determinant formula follows.
\end{proof}

Note that if $i>0$ and $B_{i+t}=0$, then it does not suffice knowing that $\cdot \overline{\ell}^{t-1}:B_i \to B_{i+t-1}$ has full rank to conclude that $\cdot \ell^t:A_i \to A_{i+t}$ has full rank, it must be surjective. For example, if $B=k[y,z]/(y,z)^3$, then $\cdot (y+z):B_1 \to B_2$ has full rank by injectivity, but $\cdot (x+y+z)^2:A_1 \to A_3$ for $A=B\otimes k[x]/(x^2)$ does not have full rank. 

Before we give an example for how Theorem~\ref{thm:recursion_quadratic} can be used, we need to recall the definition of a Hilbert series.

\begin{definition}
Let $A$ be a standard graded artinian $k$-algebra. The largest degree $D$ such that $\dim_k A_D\neq 0$ is called the \emph{socle degree} of $A$. For such an algebra of socle degree $D$, the \emph{Hilbert series} of $A$ is given by
\[
\HS(A;T) = \sum_{i=0}^{D}\dim_k (A_i) T^i.
\]
\end{definition}

In the case when $A$ has a symmetric Hilbert function, meaning that $\dim_k A_i = \dim_k A_{D-i}$ for all $i$, then it suffices to establish that all multiplication maps between these complementary degrees have full rank to show that $A$ has the strong Lefschetz property \cite[p. 104]{Lefschetz_book}. With this tool we can now use Theorem~\ref{thm:recursion_quadratic} to study one important example, namely the quadratic monomial complete intersection.

\begin{proposition}\label{prop:Quadratic_det}
Let $A=k[x_1,\dots, x_n]/(x_1^2,\dots, x_n^2)$ be a quadratic monomial complete intersection over a field of characteristic zero and let $\ell=x_1+\dots + x_n$ be a linear form. Next, denote the determinant of
\[
\cdot \ell^{n-2i}:A_i \to A_{n-i}
\]
by $a_{n,i}$. Then
\[
a_{n,i} = (-1)^{\binom{n-1}{i-1}\left(\binom{n-1}{i} + 1 \right)} \cdot a_{n-1,i} \cdot a_{n-1,i-1} \cdot (n-2i)^{\binom{n-1}{i}-\binom{n-1}{i-1}}
\]
for any $0\leq i<n/2$ with the convention that $a_{n,-1}=1$. In particular, $a_{n,i}\neq 0$ and $A$ has the strong Lefschetz property.
\end{proposition}

\begin{proof}
We begin by noting that $a_{1,0}=1\neq 0$, so we may assume that $n>1$ and that $a_{j,i}\neq 0$ for any $j<n$ and $0\leq i<j/2$. Write $A=B\otimes_k k[x_n]/(x_n^2)$ where $B=k[x_1,\dots, x_{n-1}]/(x_1^2,\dots, x_{n-1}^2)$. We know that the Hilbert series of $B$ is given by 
\[
\HS(B;T)= \sum_{j=0}^{n-1}\binom{n-1}{j}T^j.
\]
Therefore, if we set $\overline{\ell}=x_1+\cdots + x_{n-1}$, then both of the maps 
\[
\cdot \overline{\ell}^{n-1-2i}: B_i \to B_{n-1-i} \text{ and } \cdot \overline{\ell}^{n-1-2(i-1)}: B_{i-1} \to B_{n-1-(i-1)}
\]
are maps between vector spaces of the same dimensions given by $\binom{n-1}{i}$ and $\binom{n-1}{i-1}$ respectively. Hence Theorem~\ref{thm:recursion_quadratic} gives for $i\neq 0$ that 
\[
a_{n,i} = (-1)^{\binom{n-1}{i-1}\left(\binom{n-1}{i} + 1 \right)}\cdot a_{n-1,i} \cdot a_{n-1,i-1}\cdot (n-2i)^{\binom{n-1}{i}-\binom{n-1}{i-1}}
\]
as desired. When $i=0$, we also have that $B_{n}=0$, so an application of Proposition \ref{prop:quadratic_edge_case} gives that $a_{n,0}=n\cdot a_{n-1,0}$. Using our convention that $a_{n,-1}=1$, this agrees with the claimed recursion formula. Since $a_{n-1,i}\cdot a_{n-1,i-1}\neq 0$ by assumption, $a_{n,i}$ is also non-zero and we have established that $A$ has the strong Lefschetz property. 
\end{proof}

We should note that Proposition \ref{prop:Quadratic_det} is not the first result to give a recursive formula for the determinants $a_{n,i}$. A similar formula was derived by Hara and Watanabe in \cite{Watanabe_Boolean} where they studied certain matrices coming from the Boolean lattice. In their work, $a_{n,i}$ appears as the determinant of the matrix $M$ in their Proposition 6 by setting $b=i$ and where the recursion for this determinant is then given in their Theorem 4.

When the characteristic $p$ of $k$ divides $t$, Theorem \ref{thm:recursion_quadratic} can not be used. However, another recursion can be salvaged.

\begin{lemma}
\label{lem:p|t_recursion}
Pick a prime $p$ and a positive integer $t$ where $p$ divides $t$ and let $A=B\otimes_k k[x]/(x^2)$ where the characteristic of $k$ is $p$. Assume that $i>0$ and $B_{i+t}\neq 0$. Then $\cdot \ell^t:A_i \to A_{i+t}$ has full rank if and only if 
\[
\cdot \overline{\ell}^{t}: B_i \to B_{i+t} \text{ and } \cdot \overline{\ell}^{t}: B_{i-1} \to B_{i+t-1}
\]
have full rank for the same reason.
\end{lemma}

\begin{proof}
This is immediate since the matrix $M_i^t$ can be written as
\[
M_i^t=
\begin{pmatrix}
\overline{M}_i^t & 0 \\
t\overline{M}_i^{t-1} & \overline{M}_{i-1}^t
\end{pmatrix}
=
\begin{pmatrix}
\overline{M}_i^t & 0 \\
0 & \overline{M}_{i-1}^t
\end{pmatrix},
\]
where we used that $t=0$ in $k$ since $p$ divides $t$. 
\end{proof}

In the same vain as earlier, one can take care of the cases when $i=0$ or $B_{i+t}=0$. One then sees that $\cdot \ell^t:A_i\to A_{i+t}$ has full rank when $i=0$ and $B_{t}\neq 0$ if and only if $\cdot \overline{\ell}^{t-1}:B_0 \to B_{t-1}$ has full rank, when $i>0$ and $B_{i+t}=0$ if and only if $\cdot \overline{\ell}^{t}: B_{i-1} \to B_{i+t-1}$ is surjective, and when $i=0$ and $B_{i+t}=0$ it never has full rank. 

Combining the results of this section, we can now say exactly which powers of a general linear form, or equivalently which power of the sum of the variables \cite[Theorem 2.2]{Sum_of_var_SLP}, that give full rank maps in the quadratic monomial complete intersection, extending the already known results for when such an algebra has the weak or strong Lefschetz property.

\begin{theorem}
\label{thm:quadratic_criterion}
Let $A=k[x_1,\dots, x_n]/(x_1^2,\dots, x_n^2)$ be a quadratic monomial complete intersection over a field of characteristic $p$. For $i\geq 0$, $t\geq 1$ and a general linear form $\ell$, we have that $\cdot \ell^t:A_i \to A_{i+t}$ has full rank if and only if $i+t>n$ or
\[
p>\min\{i+t,n-i\}.
\]
\end{theorem}

\begin{proof}
We will prove the result by induction on $n$. If $n=1$, then any power of $\ell=x_1$ has full rank on $k[x_1]/(x_1^2)$, and since $\min\{i+t,n-i\}\leq 1$, the criterion agrees in this case. Assume now that we have some larger $n\geq 2$ and that the statement is true in $n-1$ variables. We begin by dealing with some edge cases.

If $i+t>n$, then $A_{i+t}=0$, so any map there has full rank. If instead $i+t=n$, then $A_{i+t}$ is generated by the monomial $x_1\cdots x_n$ and the map determined by $$\ell^t=(x_1+\cdots + x_n)^t=t!\left(\sum_{i_1<\dots<i_t}x_{i_1}\cdots x_{i_t}\right)$$
surjects onto this element from $A_i$ if and only if $t!\neq 0$ in $k$. That is, if $p>t$. But $i+t=n$ gives that
\[
\min\{i+t,n-i\} = \min\{n,t\} = t,
\]
so $p>t$ is exactly the criteria we require. Similarly, if $i=0$, we get that the map has full rank if and only if $\ell^t\neq 0$, which therefore happens exactly when $p>t$. Here $i=0$ gives that
\[
\min\{i+t,n-i\} = \min\{t,n\} = t,
\]
so $p>t$ is again the criteria we require.

Assume now that $i>0$, $i+t<n$ and that $p$ does not divide $t$. Then Theorem~\ref{thm:recursion_quadratic} gives that $\cdot \ell^t:A_i \to A_{i+t}$ has full rank if and only if 
\[
\cdot \overline{\ell}^{t-1}: B_i \to B_{i+t-1} \text{ and } \cdot \overline{\ell}^{t+1}: B_{i-1} \to B_{i+t}
\]
both have full rank for the same reason where $B=k[x_1,\dots, x_{n-1}]/(x_1^2,\dots, x_{n-1}^2)$. Since $\dim_k B_i = \binom{n-1}{i}$ forms a symmetric and unimodal sequence, we get that the maps must either both be injective or both be surjective if they have full rank. By induction, we therefore get that they both have full rank if and only if
\[
p>\min\{i+t-1,n-1-i\} \text{ and } p>\min\{i-1+t+1,(n-1)-(i-1)\}=\min\{i+t,n-i\}.
\]
Hence they have full rank if and only if $p>\min\{i+t,n-i\}$ as desired.

Assume now that $i>0$, $i+t<n$ and that $p$ divides $t$. By Lemma \ref{lem:p|t_recursion}, we then know that our map has full rank if and only if 
\[
\cdot \overline{\ell}^{t}: B_i \to B_{i+t} \text{ and } \cdot \overline{\ell}^{t}: B_{i-1} \to B_{i+t-1}
\]
both have full rank for the same reason. From our knowledge of the binomial coefficients we get that the only time when these may have full rank for different reasons is when the first map is only surjective and the second map is only injective, which happens exactly when $i+(i+t)=2i+t>n-1$ and $(i-1)+(i-1+t)=2i+t-2<n-1$. That is, when $2i+t=n$. But then $n-i=i+t$ and $p\leq i+t=\min\{i+t,n-1\}$ since $p$ divides $t$ and $i> 0$. Hence the criterion is correct in this case as well.

Finally, the only case left to treat is when $i>0$, $i+t<n$, $p$ divides $t$ and $i+t\neq n-i$. Then the maps 
\[
\cdot \overline{\ell}^{t}: B_i \to B_{i+t} \text{ and } \cdot \overline{\ell}^{t}: B_{i-1} \to B_{i+t-1}
\] 
both have full rank for the same reason if they both have full rank. By induction, they both have full rank if and only if
\[
p>\min\{i+t,n-1-i\} \text{ and } p>\min\{i-1+t,n-i\},
\]
which, using that $i+t\neq n-i$, happens if and only if $p>\min\{i+t,n-i\}$ and we are done.
\end{proof}

\section{The recursion for general $d$}\label{sec:recursion}

Having found how adding a new variable $x$ with $x^2=0$ effects maximal rank properties for a general linear form, let us now examine what happens when we are allowing for more non-zero powers of $x$. That is, what happens if we are adding $x$ with $x^d=0$ for some general $d\geq 2$?

Fix some artinian $k$-algebra $B$ and let 
\[
A=B\otimes_k k[x]/(x^d)
\]
where $x$ is some variable not appearing in $B$ and where $d\geq 2$. We want to determine when a power of a general linear form $\cdot \ell^t:A_i \to A_{i+t}$ has full rank in terms of powers of linear forms on $B$. To simplify the presentation of the proof, we begin with the assumptions that $i,t\geq d-1$ and $B_{i+t}\neq 0$. 

We will start similarly to how we did in the quadratic case. If we write $\ell=\overline{\ell} + x$ where $\overline{\ell}$ is a generic linear form in $B$, then
\[
\ell^t = \overline{\ell}^t + tx\overline{\ell}^{t-1} + \dots + \binom{t}{d-1}x^{d-1}\overline{\ell}^{t-(d-1)}.
\]
Next, write $\mathcal{A}_j$ for a basis of $A_j$ and $\mathcal{B}_j$ for a basis of $B_j$. Then we can choose $\mathcal{A}_j$ for $i+t\geq j\geq i$ such that
\[
\mathcal{A}_j = \mathcal{B}_{j} \sqcup x\mathcal{B}_{j-1} \sqcup \dots \sqcup x^{d-1}\mathcal{B}_{j-(d-1)}.
\]

Let $M=M_{i}^t$ be the matrix representing $\cdot \ell^t:A_i \to A_{i+t}$ in this basis and $\overline{M}_i^t$ the matrix representing $\cdot \overline{\ell}^t:B_i \to B_{i+t}$. By looking at how each term $\binom{t}{j}x^j\overline{\ell}^{t-j}$ of $\ell^t$ acts on each part of the decomposition of $\mathcal{A}_i$, we get that $M$ gets the block matrix decomposition
\begin{equation}
\label{eq:big_matrix}
M=M_i^t = 
\begin{pmatrix}
\overline{M}_i^t & 0 & \cdots & 0 & 0 \\
t\overline{M}_{i}^{t-1} & \overline{M}_{i-1}^{t} & \cdots & 0 & 0 \\
\vdots & \vdots & & \vdots & \vdots \\
\binom{t}{d-2}\overline{M}_{i}^{t-(d-2)} & \binom{t}{d-3}\overline{M}_{i-1}^{t-(d-3)} & \cdots & \overline{M}_{i-(d-2)}^{t} & 0 \\
\binom{t}{d-1}\overline{M}_{i}^{t-(d-1)} & \binom{t}{d-2}\overline{M}_{i-1}^{t-(d-2)} & \cdots & t\overline{M}_{i-(d-2)}^{t-1} & \overline{M}_{i-(d-1)}^{t}
\end{pmatrix}.
\end{equation}
In other words, $M$ takes the form of a $d\times d$ matrix where the entry in position $(p,q)$ is given by the matrix
\[
\binom{t}{p-q}\overline{M}_{i+1-q}^{t-(p-q)}
\]
if $p\geq q$, and the zero matrix of appropriate size otherwise. All references to positions, rows or columns in this section will be using this view of matrix entries unless otherwise stated.

Looking at the blocks involved, we see that all of them can be factored to include $\overline{M}_{i}^{t-(d-1)}$ as a factor. To be precise, denote 
\[
P=\overline{M}_{i}^{t-(d-1)}, \quad U_j=\overline{M}_{t-d+i+j}^{1}, \quad D_j=\overline{M}_{i-j}^{1}
\]
for $j=1,2,\dots, d-1$. Then $M$ can be written as
\[
M = 
\begin{pmatrix}
U_{d-1}U_{d-2}\cdots U_1P & 0 & \cdots & 0 & 0 \\
tU_{d-2}U_{d-3}\cdots U_1P & U_{d-2}\cdots U_1PD_1 & \cdots & 0 & 0 \\
\vdots & \vdots & & \vdots & \vdots \\
\binom{t}{d-2}U_1P & \binom{t}{d-3}U_1PD_1 & \cdots & U_1PD_1\dots D_{d-2} & 0 \\[0.3em]
\binom{t}{d-1}P & \binom{t}{d-2}PD_1 & \cdots & tPD_1\dots D_{d-2} & PD_1\dots D_{d-1}
\end{pmatrix}.
\]
We are now in a similar setting to that of Lemma \ref{lem:Block_quadratic}. Our goal is therefore to try and find square matrices $L$ and $R$ with determinant one such that $LMR$ gets an easy to understand form, namely block anti-diagonal. We make the ansatz
\[
L=
\begin{pmatrix}
I & c_{1,d-1}U_{d-1} & c_{1,d-2}U_{d-1}U_{d-2} & \cdots & c_{1,1}U_{d-1}\dots U_1 \\
0 & I & c_{2,d-2}U_{d-2} & \cdots & c_{2,1}U_{d-2}\dots U_1\\
\vdots & \vdots & \vdots & & \vdots \\
0 & 0 & 0 & \cdots & c_{d-1,1}U_1\\
0 & 0 & 0 & \cdots & I
\end{pmatrix}
\]
and 
\[
R=
\begin{pmatrix}
I & \Tilde{c}_{d-1,1}D_{1} & \Tilde{c}_{d-2,1}D_{1}D_{2} & \cdots & \Tilde{c}_{1,1}D_{1}\dots D_{d-1} \\
0 & I & \Tilde{c}_{d-2,2}D_{2} & \cdots & \Tilde{c}_{1,2}D_{2}\dots D_{d-1}\\
\vdots & \vdots & \vdots & & \vdots \\
0 & 0 & 0 & \cdots & \Tilde{c}_{1,d-1}D_{d-1}\\
0 & 0 & 0 & \cdots & I
\end{pmatrix}
\]
where $c_{i,j}$ and $\Tilde{c}_{i,j}$ are some coefficients that are yet to be determined. 

\begin{lemma}\label{lem:well_defined_L&R}
The matrices $L$ and $R$ are well defined square matrices with determinant one.
\end{lemma}

\begin{proof}
From the definition of $U_j=\overline{M}_{t-d+i+j}^{1}$ and $D_j=\overline{M}_{i-j}^{1}$, we know that $U_{d-1}\dots U_1$ and $D_{1}\dots D_{d-1}$ are well defined matrices. The fact that the rest of the matrices in $L$ and $R$ are well defined then follows. Next, let $r(U_j)$ denote the number of rows of $U_j$ and $c(U_j)$ the number of columns of $U_j$. Then 
\[
r(L)=r(U_{d-1}) + r(U_{d-2}) + \cdots + r(U_{1}) + c(U_1)
\]
where $c(U_1)$ comes from the fact the the size of the identity matrix appearing in the bottom right corner of $L$ has size $c(U_1)$. We also have that 
\[
c(L)=r(U_{d-1}) + c(U_{d-1}) + c(U_{d-2}) + \cdots + c(U_1)
\]
where $r(U_{d-1})$ comes from the fact the the size of the identity matrix appearing in the top left corner of $L$ has size $r(U_{d-1})$. But we know that $c(U_j)=r(U_{j-1})$ for $j=2,3,\dots, d-1$ from the definition of $U_j$, and these equalities gives that $r(L)=c(L)$ and $L$ is a square matrix. The argument for $R$ is analogous. Finally, $\det(L)=\det(R)=1$ since both are upper triangular square matrices with ones along the diagonal.
\end{proof}

The task now is to see if we can choose the constants $c_{i,j}$ and $\Tilde{c}_{i,j}$ such that $LMR$ gets our desired form. If we view all of $L, M, R$ as $d\times d$ matrices with matrix entries, then in position $(p,q)$, we find that $LM$ is given by
\begin{equation}
\label{eq:LM}
(LM)_{p,q} = \left(\binom{t}{p-q} + \sum_{j=p}^{d-1}c_{p,d-j}\binom{t}{j+1-q}\right)U_{d-p}\dots U_1PD_1\dots D_{q-1}.
\end{equation}
In particular, there is only one kind of block appearing in each position, just different scalar multiples of it. This is one motivation for our choice of the matrices $L$ (and $R$). Continuing the multiplication by $R$, we get
\begin{equation}
\label{eq:LMR}
(LMR)_{p,q} = (LM)_{p,q} + \sum_{j=1}^{q-1}\Tilde{c}_{d+1-q,j}(LM)_{p,j}D_j\dots D_{q-1}.
\end{equation}
Note that there is only one kind of block appearing in both $(LM)_{p,q}$ and the matrices $(LM)_{p,j}D_j\dots D_{q-1}$ for any $j$, namely scalar multiples of $U_{d-p}\dots U_1 P D_1\dots D_{q-1}$. Therefore we may talk about the scalar in front of $U_{d-p}\dots U_1 P D_1\dots D_{q-1}$ at position $(p,q)$ in $LM$ or $LMR$. Denote this scalar by 
\[
[(LM)_{p,q}] = \binom{t}{p-q} + \sum_{j=p}^{d-1}c_{p,d-j}\binom{t}{j+1-q}
\]
and
\[
[(LMR)_{p,q}] = [(LM)_{p,q}] + \sum_{j=1}^{q-1}\Tilde{c}_{d+1-q,j}[(LM)_{p,j}]
\]
respectively. 

We will now see how to pick $c_{i,j}$ and $\Tilde{c}_{i,j}$ such that $(LMR)_{p,q}=0$ if $p+q\neq d+1$, thus achieving the desired block anti-diagonal form. But before that, we observe that there is symmetry in our problem.

\begin{lemma}\label{lem:tilde(c)=c}
If we can choose only the constants $c_{i,j}$ to ensure that $(LMR)_{p,q}=0$ if $p+q<d+1$, then we can choose $\Tilde{c}_{i,j}$ such that $(LMR)_{p,q}=0$ if $p+q>d+1$. 
\end{lemma}

\begin{proof}
From Equation \eqref{eq:LMR}, we see that it suffices to choose $c_{i,j}$ such that $(LM)_{p,q}=0$ for all $p+q<d+1$ to ensure that $(LMR)_{p,q}=0$ for all such $p$ and $q$. Next, by first considering $MR$, we have that
\[
(MR)_{p,q}= \left(\binom{t}{p-q} + \sum_{i=1}^{q-1}\Tilde{c}_{d+1-q,i}\binom{t}{p-i}\right)U_{d-p}\dots U_1PD_1\dots D_{q-1}
\]
and moreover that
\[
(LMR)_{p,q} = (MR)_{p,q} + \sum_{j=p}^{d-1}c_{p,d-j}U_{d-p}\cdots U_{d-j}(MR)_{j+1,q}.
\]
So if $(MR)_{p,q}=0$ for all $p,q>d+1$, then $(LMR)_{p,q}=0$ for all such $p$ and $q$. We now note that
\begin{align*}
[(MR)_{d+1-q,d+1-p}] &= \binom{t}{p-q} + \sum_{i=1}^{d-p}\Tilde{c}_{p,i}\binom{t}{d+1-q-i} \\
&=\binom{t}{p-q} + \sum_{j=p}^{d-1}\Tilde{c}_{p,d-j}\binom{t}{j+1-q},
\end{align*}
which is the same expression as we have for $[(LM)_{p,q}]$ in \eqref{eq:LM} but with $\Tilde{c}_{i,j}$ instead of $c_{i,j}$. Since $p+q<d+1$ if and only if $(d+1-q)+(d+1-p)>d+1$, we therefore get that if we can find $c_{i,j}$ such that $(LMR)_{p,q}=0$ for all $p+q<d+1$, then picking $\Tilde{c}_{i,j}=c_{i,j}$ for all $i,j$ gives that $(LMR)_{p,q}=0$ for all $p+q>d+1$ as well.
\end{proof}

Lemma \ref{lem:tilde(c)=c} raises the question if we can choose the constants $c_{i,j}$ such that $(LMR)_{p,q}=0$ when $p+q<d+1$. The answer here might depend on the characteristic of our field $k$, but it is definitely yes in characteristic zero.

\begin{lemma}\label{lem:Solve_system}
Assume that the characteristic of $k$ is either zero or a prime that does not divide the determinant of
\[
C_{t,p,d} = \begin{pmatrix}
\binom{t}{d-1} & \binom{t}{d-2} & \cdots & \binom{t}{p} \\[0.5em]
\binom{t}{d-2} & \binom{t}{d-3} & \cdots & \binom{t}{p-1} \\
\vdots & \vdots & & \vdots \\
\binom{t}{p} & \binom{t}{p-1} & \cdots & \binom{t}{2p-(d-1)}
\end{pmatrix}.
\]
for any $p<d$. Then we can choose the constants $c_{i,j}$ and $\Tilde{c}_{i,j}$ such that $(LMR)_{p,q}=0$ if $p+q\neq d+1$. 
\end{lemma}

\begin{proof}
By Lemma \ref{lem:tilde(c)=c} and Equation \eqref{eq:LMR}, it suffices to ensure that we can find $c_{i,j}$ such that $(LM)_{p,q}=0$ if $p+q<d+1$. Fix some row $p$ where $1\leq p \leq d-1$. We then require that $(LM)_{p,q}=0$ for all $q$ in the range $1\leq q \leq d-p$. By Equation~\eqref{eq:LM}, this gives a system of equations in the unknowns $c_{p,1},\dots, c_{p,d-p}$ of the form
\begin{equation}
\label{eq:matrix equation}
\begin{pmatrix}
\binom{t}{d-1} & \binom{t}{d-2} & \cdots & \binom{t}{p} \\[0.5em]
\binom{t}{d-2} & \binom{t}{d-3} & \cdots & \binom{t}{p-1} \\
\vdots & \vdots & & \vdots \\
\binom{t}{p} & \binom{t}{p-1} & \cdots & \binom{t}{2p-(d-1)}
\end{pmatrix} 
\begin{pmatrix}
c_{p,1}\\
c_{p,2} \\
\vdots \\
c_{p,d-p}
\end{pmatrix} = -
\begin{pmatrix}
\binom{t}{p-1} \\[0.5em]
\binom{t}{p-2} \\
\vdots \\
\binom{t}{2p-d}
\end{pmatrix}.
\end{equation}
Hence, if the characteristic of $k$ is a prime that does not divide the determinant of $C_{t,p,d}$, then the system has a solution as desired. We further note that $\det(C_{t,p,d})$ is known up to a sign. For example, we have the closed formula
\begin{equation}\label{eq:closed_determinant}
|\det (C_{t,p,d})| = \prod_{i=1}^{d-p}\prod_{j=1}^{p}\prod_{k=1}^{t-p}\frac{i+j+k-1}{i+j+k-2}
\end{equation}
given by Krattenthaler \cite[Equation 2.17]{Krattenthaler1999} since up to reordering of the columns, $C_{t,p,d}=M_{d-p}(p,t-p)$ in Krattenthaler's notation. In particular, $\det (C_{t,p,d})\neq 0$, so the system has a solution if the characteristic of $k$ is zero.
\end{proof}

Assume now that we are in a case where Lemma \ref{lem:Solve_system} applies, such as when our field $k$ has characteristic zero. Then we have shown that we can find matrices $L$ and $R$ such that $\det(M)=\det(LMR)$ and the only non-zero blocks of $LMR$ are along the anti-diagonal $p+q=d+1$. Moreover, using that also $(LM)_{p,j}=0$ if $j<d+1-p$, Equation \eqref{eq:LMR} gives that blocks on the anti-diagonal satisfy $(LMR)_{p,d+1-p} = (LM)_{p,d+1-p}$, and Equation \eqref{eq:LM} that $(LM)_{p,d+1-p}$ is given by
\[
\left(\binom{t}{2p-(d+1)} + \sum_{j=p}^{d-1}c_{p,d-j}\binom{t}{j+p-d}\right)U_{d-p}\dots U_1PD_1\dots D_{d-p}.
\]
Since we understand $U_{d-p}\dots U_1PD_1\dots D_{d-p}$ recursively as it represents the multiplication 
\[
\cdot \overline{\ell}^{d+1+t-2p}:B_{i-d+p}\to B_{t+i-p+1}
\]
in $B$, we need to understand $[(LM)_{p,d+1-p}]$. Recalling that the constants $c_{p,d-j}$ are obtained as solutions to the system \eqref{eq:matrix equation}, we can give the following nice expressions for $[(LM)_{p,d+1-p}]$.

\begin{proposition}
\label{prop:diagonal_coeff}
The coefficients $[(LM)_{p,d+1-p}]$ appearing along the anti-diagonal of $LMR$ are given by
\[
[(LM)_{p,d+1-p}] = \frac{\det (C_{t,p-1,d})}{\det(C_{t,p,d})}
\]
for $p=1,2,\dots, d-1$.
\end{proposition}

\begin{proof}
By Equation \eqref{eq:LMR}, we have that
\begin{equation}\label{eq:LM_coeff}
[(LM)_{p,d+1-p}] =\binom{t}{2p-(d+1)} + \sum_{j=1}^{d-p}c_{p,d-j}\binom{t}{p-j}.
\end{equation}

Next, we know that the $c_{p,j}$ are given as the solutions of the system \eqref{eq:matrix equation}. Cramer's rule then gives us that
\[
c_{p,j}=\frac{\det (C^{(j)}_{t,p,d})}{\det(C_{t,p,d})}
\]
where $C^{(j)}_{t,p,d}$ denotes the matrix obtained from $C_{t,p,d}$ by swapping out column $j$ in $C_{t,p,d}$ with 
\[ v_{p}=
-\begin{pmatrix}
\binom{t}{p-1} \\
\binom{t}{p-2} \\
\vdots \\
\binom{t}{2p-d}
\end{pmatrix}.
\]
Inserting this expression for $c_{p,j}$ into Equation \eqref{eq:LM_coeff} gives 
\begin{equation}\label{eq:LM_coeff_step2}
\det(C_{t,p,d})[LM_{p,d+1-p}] = \det(C_{t,p,d})\binom{t}{2p-(d+1)} + \sum_{j=1}^{d-p}\binom{t}{p-j}\det(C^{(j)}_{t,p,d}).
\end{equation}
To prove the proposition, we are thus left to show that the right hand side of Equation \eqref{eq:LM_coeff_step2} equals $\det(C_{t,p-1,d})$. To this end, we make two observations. Firstly, $C_{p,d}$ sits as the submatrix of $C_{p-1,d}$ obtained by removing the last row and column. Secondly, removing only the last row of $C_{p-1,d}$, the last column in the remaining matrix is exactly $-v_p$. Therefore, expanding $\det(C_{p-1,d})$ along the last row, we get that
\begin{equation}
\label{eq:LM_coeff_step3}
\det(C_{t,p-1,d}) = \sum_{j=1}^{d-(p-1)}(-1)^{d-p+1+j}\binom{t}{p-j}\det(\Tilde{C}_{t,p,d}^{(j)})
\end{equation}
where $\Tilde{C}_{t,p,d}^{(j)}$ for $j\neq d-(p-1)$ is the matrix obtained from $C_{t,p,d}$ by deleting column $j$ and inserting $-v_p$ as the last column, and $\Tilde{C}_{t,p,d}^{(d-(p-1))}=C_{t,p,d}$. Hence $\det(\Tilde{C}_{t,p,d}^{(j)})$ differs from $\det({C}^{(j)}_{t,p,d})$ by a factor of $(-1)^{d-p-j+1}$ coming from $d-p-j$ row switches and the difference of $-v_p$ and $v_p$. In other words, 
\[
\det(\Tilde{C}_{t,p,d}^{(j)}) = (-1)^{d-p-j+1}\det({C}^{(j)}_{t,p,d}).
\]
Combining this expression with Equation \eqref{eq:LM_coeff_step3} gives
\begin{align*}
\det(C_{t,p-1,d}) &= \sum_{j=1}^{d-(p-1)}(-1)^{2(d-p+1)}\binom{t}{p-j}\det({C}_{t,p,d}^{(j)}) \\
&=\det(C_{t,p,d})\binom{t}{2p-(d+1)} + \sum_{j=1}^{d-p}\binom{t}{p-j}\det({C}_{t,p,d}^{(j)}),  
\end{align*}
which is exactly the right hand side of Equation \eqref{eq:LM_coeff_step2}.
\end{proof}

We further note that $[(LM)_{d,1}]=\binom{t}{d-1}$ by Equation \eqref{eq:LM} and that $|\det(C_{t,0,d})|=1$ since $C_{t,0,d}$ has ones along the anti-diagonal and zeroes below it. Combining everything we have done so far, we have arrived at the first version of our main theorem.

\begin{theorem}
\label{thm:Rank-recursion_v1}
Let $A=B\otimes_k k[x]/(x^d)$ for some $d\geq 1$ where $B$ is an artinian $k$-algebra. Assume the characteristic of $k$ is either zero or a prime which does not divide $\prod_{p=1}^{d-1}\det(C_{t,p,d})$ where
\[
C_{t,p,d} = \begin{pmatrix}
\binom{t}{d-1} & \binom{t}{d-2} & \cdots & \binom{t}{p} \\[0.5em]
\binom{t}{d-2} & \binom{t}{d-3} & \cdots & \binom{t}{p-1} \\
\vdots & \vdots & & \vdots \\
\binom{t}{p} & \binom{t}{p-1} & \cdots & \binom{t}{2p-(d-1)}
\end{pmatrix}.
\]
If $i,t\geq d-1$ and $B_{i+t}\neq 0$, then the multiplication $\cdot \ell^t:A_i \to A_{i+t}$ for a general linear form $\ell = \overline{\ell} + x\in A_1$ has full rank if and only if all the maps 
\[
\cdot \overline{\ell}^{2q + t -(d-1)}:B_{i-q} \to B_{i+q+t-(d-1)}
\]
have full rank for the same reason for $q$ from $0$ to $d-1$. Moreover, if the matrices $\overline{M}_{i-q}^{2q+t-(d-1)}$ representing these maps are all square, then the matrix $M_i^t$ representing $\cdot \ell^t$ is square with
\begin{align*}
|\det(M_{i}^t)| = \left(\prod_{q=0}^{d-1}|\det(\overline{M}_{i-q}^{2q+t-(d-1)})| \right)\prod_{p=1}^{d-1}|\det(C_{t,p,d})|^{\dim_k(B_{i+p-(d-1)})-\dim_k(B_{i+p-d})}.
\end{align*}
\end{theorem}

\begin{proof}
From Lemma \ref{lem:well_defined_L&R} and Lemma \ref{lem:Solve_system}, we know that under the assumptions of the theorem, there are invertible matrices $L$ and $R$ with determinant one such that $LM_i^tR$ can be written as a block anti-diagonal matrix with $d$ blocks. By using Equation \eqref{eq:LM}, Equation \eqref{eq:LMR} and Proposition \ref{prop:diagonal_coeff}, these blocks are then given by
\[
\frac{\det (C_{t,p-1,d})}{\det(C_{t,p,d})}U_{d-p}\cdots U_1PD_1\cdots D_{d-p}
\]
for $p=1,2,\dots, d-1$ and $\binom{t}{d-1}P$ for the last block. Since each of the coefficients in front of $U_{d-p}\cdots U_1PD_1\cdots D_{d-p}$ is non-zero, we get that $LM_i^tR$, and therefore also $M_i^t$, has full rank if and only if all of the matrices $U_{d-p}\cdots U_1PD_1\cdots D_{d-p}$ for $p=1,2,\dots, d$ have full rank for the same reason. Moreover, as $U_{d-p}\cdots U_1PD_1\cdots D_{d-p}$ is a matrix representing the map
\[
\cdot \overline{\ell}^{d+1+t-2p}:B_{i-d+p}\to B_{t+i-p+1}
\]
in $B$, a reparametrization via $q=d-p$ gives that $M_i^t$ has full rank if and only if all the maps 
\[
\cdot \overline{\ell}^{2q + t -(d-1)}:B_{i-q} \to B_{i+q+t-(d-1)}
\]
have full rank for the same reason for $q$ from $0$ to $d-1$. In the case when all matrices $\overline{M}_{i-q}^{2q+t-(d-1)}$ are square, then the formula for $|\det(M_i^t)|$ follows as the product of all determinants along the anti-diagonal. Here we just note that the exponents come from the sizes of the blocks determined by each $\overline{M}_{i-q}^{2q+t-(d-1)}$ and that $|\det(C_{t,0,d})|=1$. 
\end{proof}

Before exploring applications of this result, we wish to get rid of the constraints $i,t\geq d-1$ and $B_{i+t}\neq 0$. Let us begin by relaxing the constraint on $t$. If $t<d-1$, then our decomposition of $\ell^t$ is changed and becomes
\[
\ell^t = \overline{\ell}^t + tx\overline{\ell}^{t-1} + \dots + \binom{t}{t}x^{t}\overline{\ell}^{0}.
\]
Therefore, some of the blocks at the bottom left corner of $M$ vanishes and the matrix instead takes the form
\[
\begin{pmatrix}
\overline{M}_i^t & 0 & \cdots & 0 & 0 & 0 & \cdots & 0 & 0 \\
t\overline{M}_{i}^{t-1} & \overline{M}_{i-1}^{t} & \cdots & 0 & 0 & 0 & \cdots & 0 & 0 \\
\vdots & \vdots & & & & & & \vdots & \vdots \\
\binom{t}{t-1}\overline{M}_{i}^{1} & \binom{t}{t-2}\overline{M}_{i-1}^{2} & \cdots & & & & & 0 & 0\\[0.3em]
I & \binom{t}{t-1}\overline{M}_{i-1}^{1} & \cdots & & & & & 0 & 0\\
0 & I & \cdots & & & & & 0 & 0\\
\vdots & \vdots & & \vdots & \vdots & \vdots & & \vdots & \vdots \\
0 & 0 & \cdots & I & \binom{t}{t-1}\overline{M}_{i-t}^{1} & \binom{t}{t-2}\overline{M}_{i-t-1}^{2} & \cdots & \overline{M}_{i-(d-2)}^{t} & 0 \\[0.3em]
0 & 0 & \cdots & 0 & I & \binom{t}{t-1}\overline{M}_{i-t-1}^{1} & \cdots & t\overline{M}_{i-(d-2)}^{t-1} & \overline{M}_{i-(d-1)}^{t}
\end{pmatrix}.
\]

That is, $M$ is given just as before but with the change that if we view $M$ as a $d\times d$ matrix with matrix entries, then the entries in position $(p,q)$ with $p-q=t$ is now given by $I$ and the entries where $p-q>t$ are zero. 

Hence, if we take the same matrices $L$ and $R$ as before, a short calculation shows that $(LMR)_{p,q}$ is still zero if $p-q>t$ and the identity matrix if $p-q=t$, while it gives that same expressions as in the case when $t\geq d-1$ if $p\leq t$ or $q\geq d-(t-1)$. Thus the first $t$ rows and last $t$ columns of $LMR$ are zero except for the non-zero blocks on the diagonal that we understand from Proposition \ref{prop:diagonal_coeff}. But as we can view the bottom left $(d-t)\times (d-t)$ blocks of $LMR$ as one large square block that is upper triangular with ones on the diagonal, that does not cause any additional constraints. 

We could now write down a generalization of Theorem \ref{thm:Rank-recursion_v1} that allows for any value of $t$, but let us do that later after we have taken care of the remaining constraints. To that end, remove all the constraints from Theorem \ref{thm:Rank-recursion_v1}, so it might be that $i<d-1$ or $B_{i+t}=0$. The main difference from the earlier cases is now that the number of parts in our decomposition of the basis $\mathcal{A}_j$ of $A_j$ may depend on $j$. This time we have that
\[
\mathcal{A}_j = x^{\max\{0,j-s\}}\mathcal{B}_{\min\{j,s\}} \sqcup x^{\max\{0,j-s\} + 1}\mathcal{B}_{\min\{j,s\}-1} \sqcup \cdots \sqcup x^{\min\{d-1,j\}}\mathcal{B}_{\max\{0,j-(d-1)\}}
\]
where $s$ denotes the socle degree of $B$.
Looking at the decompositions for $\mathcal{A}_i$ and $\mathcal{A}_{i+t}$, we get that the matrix $M$ representing $\cdot \ell^t:A_i \to A_{i+t}$ takes the form
\[
M=
\begin{pmatrix}
\binom{t}{v}\overline{M}_{\min\{i,s\}}^v & \binom{t}{v+1}\overline{M}_{\min\{i,s\}-1}^{v+1} & \cdots & \binom{t}{w}\overline{M}_{\max\{0,i-(d-1)\}}^w \\[0.5em]
\binom{t}{v-1}\overline{M}_{\min\{i,s\}}^{v-1} & \binom{t}{v}\overline{M}_{\min\{i,s\}-1}^{v} & \cdots & \binom{t}{w-1}\overline{M}_{\max\{0,i-(d-1)\}}^{w-1} \\
\vdots & \vdots & & \vdots \\
\binom{t}{u}\overline{M}_{\min\{i,s\}}^{u} & \binom{t}{u+1}\overline{M}_{\min\{i,s\}-1}^{u+1} & \cdots & \binom{t}{u+w-v}\overline{M}_{\max\{0,i-(d-1)\}}^{u+w-v}
\end{pmatrix}
\]
where
\begin{align*}
v &= \min\{i+t,s\}-\min\{i,s\},\\ 
u &= \max\{0,i+t-(d-1)\}-\min\{i,s\},\\ 
w &= \min\{i+t,s\}-\max\{0,i-(d-1)\}
\end{align*}
and $\binom{t}{j}=0$ if $j<0$ or $j>t$. While this might look like several cases that needs to be taken care of, some of them are easy to handle. First, if $\min\{i,s\}=s$, then $v=0$. This gives that the diagonal starting in the top left corner of $M$ consists of identity matrices with zeroes below it and thus that $M$ has full rank. Similarly, if $\max\{0,i-(d-1)\}=0$, then $u+w-v=0$ and we get that the diagonal starting at the bottom right corner of $M$ consists of identity matrices with zeroes below it, so $M$ has full rank in this case as well.

We may therefore assume that $i<s$ and $i+t>d-1$ from now on. With those assumptions, $M$ takes the slightly easier form
\[
M=
\begin{pmatrix}
\binom{t}{v}\overline{M}_{i}^v & \binom{t}{v+1}\overline{M}_{i-1}^{v+1} & \cdots & \binom{t}{w}\overline{M}_{\max\{0,i-(d-1)\}}^w \\[0.5em]
\binom{t}{v-1}\overline{M}_{i}^{v-1} & \binom{t}{v}\overline{M}_{i-1}^{v} & \cdots & \binom{t}{w-1}\overline{M}_{\max\{0,i-(d-1)\}}^{w-1} \\
\vdots & \vdots & & \vdots \\
\binom{t}{d-1}\overline{M}_{i}^{t-(d-1)} & \binom{t}{d-2}\overline{M}_{i-1}^{t-(d-1)+1} & \cdots & \binom{t}{w-v+t-(d-1)}\overline{M}_{\max\{0,i-(d-1)\}}^{w-v+t-(d-1)}
\end{pmatrix}
\]
where $v=\min\{i+t,s\}-i$ and $w=\min\{i+t,s\}-\max\{0,i-(d-1)\}$. If we denote by $M_{\text{big}}$ the matrix which we called $M$ before in Equation \eqref{eq:big_matrix}, then we note that our current matrix $M$ sits as the submatrix of $M_{\text{big}}$ consisting of the first $i-\max\{0,i-(d-1)\} + 1$ columns and last $d-i-t+\min\{i+t,s\}$ rows of $M_{\text{big}}$. Let $L_{\text{big}}$ and $R_{\text{big}}$ be the corresponding matrices to $M_{\text{big}}$ such that $L_{\text{big}}M_{\text{big}}R_{\text{big}}$ gets only non-zero blocks on the anti-diagonal as before. If $L$ denotes the square block of size $d-i-t+\min\{i+t,s\}$ at the bottom right of $L_{\text{big}}$ and $R$ the square block of size $i-\max\{0,i-(d-1)\} + 1$ at the top left of $R_{\text{big}}$, then we get that
\[
L_{\text{big}}M_{\text{big}}R_{\text{big}} = 
\begin{pmatrix}
* & * \\
0 & L
\end{pmatrix}
\begin{pmatrix}
* & * \\
M & *
\end{pmatrix}
\begin{pmatrix}
R & * \\
0 & *
\end{pmatrix} 
=
\begin{pmatrix}
* & * \\
LMR & *
\end{pmatrix}
\]
where $*$ denotes a block of appropriate size that is not relevant to us. Thus $LMR$ is the matrix of the same size as $M$ which lies at the bottom left corner of $L_{\text{big}}M_{\text{big}}R_{\text{big}}$. With this, we are now ready to give the final version of our main theorem. \newpage

\begin{theorem}
\label{thm:Rank-recursion}
Let $A=B\otimes_k k[x]/(x^d)$ for some $d\geq 1$ where $B$ is an artinian $k$-algebra with socle degree $s$. Assume the characteristic of $k$ is either zero or a prime which does not divide
\begin{equation}
\label{eq:prime_cond}
\prod_{p=\max\{1,d-i,i+t-s+1\}}^{\min\{d-1,t-1\}}\det(C_{t,p,d}) 
\end{equation}
where
\[
C_{t,p,d} = \begin{pmatrix}
\binom{t}{d-1} & \binom{t}{d-2} & \cdots & \binom{t}{p} \\[0.5em]
\binom{t}{d-2} & \binom{t}{d-3} & \cdots & \binom{t}{p-1} \\
\vdots & \vdots & & \vdots \\
\binom{t}{p} & \binom{t}{p-1} & \cdots & \binom{t}{2p-(d-1)}
\end{pmatrix}.
\]
If $\overline{\ell}$ is a general linear form in $B$ and $\ell=\overline{\ell} + x$ is a general linear form in $A$, then the map $\ell^t:A_i \to A_{i+t}$ has full rank if and only if the maps
\[
\cdot \overline{\ell}^{2q + t -(d-1)}:B_{i-q} \to B_{i+q+t-(d-1)}
\]
have full rank for the same reason for $q$ from $\max\{0,d-t\}$ to $d-1$.

Moreover, if the matrices $\overline{M}_{i-q}^{2q+t-(d-1)}$ representing these maps for $q$ going from $\max\{0,d-t\}$ to $\min\{d-1,d-1-i-t+s,i\}$ and the matrix $M$ representing $\cdot \ell^t$ are all square, then $$|\det(M)|=F_1\cdot F_2 \cdot F_3$$ where
\begin{align*}
F_1&=\prod_{q=\max\{0,d-t\}}^{\min\{d-1,d-1-i-t+s,i\}}|\det(\overline{M}_{i-q}^{2q+t-(d-1)})|, \\
F_2&= \prod_{p=\max\{1,1+i+t-s,d-i-1\}}^{\min\{d-1,t-1\}}|\det(C_{t,p,d})|^{\dim_k(B_{i+p-(d-1)})-\dim_k(B_{i+p-d})}\\
F_3&=\begin{cases}
|\det(C_{t,i+t-s,d})|^{\dim_k(B_{2i-(d-1)+t-s})} & \text{if } \eqref{eq:F_3_crit}\text{ holds,} \\
1 & \text{otherwise,}
\end{cases}
\end{align*}
where \eqref{eq:F_3_crit} is the inequality
\begin{equation}
\label{eq:F_3_crit}
\min\{d-1,t-1\}\geq i+t-s\geq\max\{0,(d-1)-i\}.
\end{equation}
\end{theorem}

\begin{proof}
Assume first that we are working in characteristic zero. By the above discussion, we then have that $\cdot \ell^t:A_i \to A_{i+t}$ has full rank if and only if the maps $\cdot \overline{\ell}^{2q + t -(d-1)}:B_{i-q} \to B_{i+q+t-(d-1)}$ satisfy some full rank conditions for a subset of $q$ in the range from $0$ to $d-1$. This is the case since all $q$ are part of the recursion when $i,t\geq d-1$ and $s\leq i+t$ by Theorem \ref{thm:Rank-recursion_v1}, while our discussion in the case when $t<d-1$ gives that some of the blocks of $LMR$ will no longer be present, namely those corresponding to $q<d-t$. This gives the part saying that $q$ goes from $\max\{0,d-t\}$.

Next, in the case without any assumptions, we saw that we are only getting a contribution from the blocks lying in the first $i-\max\{0,i-(d-1)\} + 1$ columns and last $d-i-t+\min\{i+t,s\}$ rows of $L_{\text{big}}M_{\text{big}}R_{\text{big}}$. Recalling that each $q$ corresponds to the non-zero block in column $q+1$ of $L_{\text{big}}M_{\text{big}}R_{\text{big}}$ gives that $q$ can also not be larger than $i-\max\{0,i-(d-1)\}$ or $d-i-t+\min\{i+t,s\}-1$. Since $q$ is always at most $d-1$, writing out the four cases here gives that $q$ will be at most $\min\{d-1,d-1-i-t+s,i\}$. Further, let $\mathcal{L}$ denote the collection of maps given by
\[
\cdot \overline{\ell}^{2q + t -(d-1)}:B_{i-q} \to B_{i+q+t-(d-1)}
\]
where $q$ goes from $\max\{0,d-t\}$ to $\min\{d-1,d-1-i-t+s,i\}$. For $\cdot \ell^t:A_i \to A_{i+t}$ to have full rank, it is sometimes required that the maps in $\mathcal{L}$ have full rank for specific reasons. If $i<\min\{d-1,d-1-i-t+s\}$, then $LMR$ has a top row consisting of blocks of zeroes. Thus it can only represent an injective map, not a surjective map. Therefore, all maps in $\mathcal{L}$ must be injective for $LMR$ (and thus also $M$) to have full rank. Then $q=i+1$ gives a map from $B_{-1}=0$ to $B_{2i+t-(d-1)}\neq 0$ that is trivially injective but not surjective. Hence we may include $q>i$ in our recursion and let the requirement that all maps should have full rank force the injectivity. In the same way, if $d-1-i-t+s<\min\{d-1,i\}$, then $LMR$ has the last column filled with blocks of zeroes, so it can only represent a surjective map. Thus all maps in $\mathcal{L}$ must be surjective for $LMR$ to have full rank. Using $q=d-i-t+s$, it gives a map from a non-zero space to a space of dimension zero, so adding $q>d-1-i-t+s$ to the recursion will force surjectivity. Finally, in any other case there is no automatic row or column of zeroes since the anti-diagonal blocks of $LMR$ representing maps in $\mathcal{L}$ do reach the top right corner of $LMR$. Hence the maps in $\mathcal{L}$ can be either all injective or all surjective for $LMR$ to have full rank.

Note also that the interval from $\max\{0,d-t\}$ to $\min\{d-1,d-1-i-t+s,i\}$ is empty if either $i\geq s$ or $d-1\geq i+t$. Since the addition of maps coming from $q>\min\{d-1,d-1-i-t+s,i\}\neq d-1$ only adds map that have full rank for trivial reasons, the recursion says that $\cdot \ell^t:A_i \to A_{i+t}$ always have full rank if either $i\geq s$ or $d-1\geq i+t$, which agrees with what was determined earlier. 

In positive characteristic, the thing we need to make sure of is that all matrices $C_{t,p,d}$ used when constructing the matrices $L$ and $R$ are invertible. From the proofs of Lemma \ref{lem:tilde(c)=c} and Lemma \ref{lem:Solve_system}, we see that each $p$ corresponds to making sure that all entries of $M_{\text{big}}$ in row $p$ are zero to the left of the anti-diagonal and all entries in column $d+1-p$ are zero below the anti-diagonal. When $t<d-1$, we got that only the first $t$ rows and last $t$ columns needed input from $L$ and $R$, hence we only need $C_{t,p,d}$ with $p\leq t$ to be invertible. But since $|\det(C_{t,t,d})|=1$, we may even say $p<t$. Similarly, in the general case we had that only the first $i-\max\{0,i-(d-1)\} +1$ columns and last $d-i-t+\min\{i+t,s\}$ rows of $M_{\text{big}}$ are part of the matrix representing $\cdot \ell^t$. Hence only $p$ with $p>i+t-\min\{i+t,s\}$ and $p\geq d-i+\max\{0,i-(d-1)\}$ must be considered, or in other words, $p\geq \max\{1,d-i,i+t-s+1\}$, giving the claimed bound.

In the case when all the matrices are square, the formula for $|\det(M)|$ follows by taking the product of the determinants of all matrices appearing on the anti-diagonal of $LMR$. Since those matrices are exactly 
\[
\frac{\det (C_{t,d-1-q,d})}{\det(C_{t,d-q,d})}\overline{M}_{i-q}^{2q+t-(d-1)}
\]
for $q$ from $\max\{0,d-t\}$ to $\min\{d-1,d-1-i-t+s,i\}$, where we set $\det(C_{t,d,d})=1$, the formula for $F_1$ follows. For $F_2$, we do the reparametrization $p=d-q$ to get the factors
\begin{align*}
&|\det(C_{t,\max\{0,i+t-s,(d-1)-i\},d})|^{\dim_k (B_{\max\{i-(d-1),2i-(d-1)+t-s,0\}})}, \\
&|\det(C_{t,\min\{d,t\},d})|^{-\dim_k(B_{\min\{i,i+t-d\}})}
\end{align*}
and
\[
\prod_{p=\max\{1,1+i+t-s,d-i\}}^{\min\{d-1,t-1\}}|\det(C_{t,p,d})|^{\dim_k(B_{i+p-(d-1)})-\dim_k(B_{i+p-d})},
\]
where the first two factors are only included if the last product is not empty. Here the second factor does not contribute since $|\det(C_{t,t,d})|=|\det(C_{t,d,d})|=1$. For the first factor we see that it will not contribute anything when $\max\{0,i+t-s,(d-1)-i\}=0$ since $|\det(C_{t,0,d})|=1$. However, if $\min\{d-1,t-1\}\geq \max\{0,i+t-s,(d-1)-i\}>0$, it will contribute a factor. When $\max\{0,i+t-s,(d-1)-i\} = (d-1)-i$, we can use that since $\dim_k B_{i+p-d} = \dim_k B_{-1} = 0$ for $p=(d-1)-i$, we may write it in such a way as to include it in $F_2$ without changing the value, giving the total 
\[
F_2=\prod_{p=\max\{1,1+i+t-s,d-i-1\}}^{\min\{d-1,t-1\}}|\det(C_{t,p,d})|^{\dim_k(B_{i+p-(d-1)})-\dim_k(B_{i+p-d})}
\]
and 
if $\min\{d-1,t-1\}\geq i+t-s\geq\max\{0,(d-1)-i\}$, also the factor 
\[
F_3=|\det(C_{t,i+t-s,d})|^{\dim_k(B_{2i-(d-1)+t-s})}
\]
as claimed. 
\end{proof}

\begin{remark}
From the proof of Theorem \ref{thm:Rank-recursion}, we see that if $\overline{\ell}\in B_1$ is a general linear form and $\mathcal{L}$ denotes the collection of maps given by
\[
\cdot \overline{\ell}^{2q + t -(d-1)}:B_{i-q} \to B_{i+q+t-(d-1)}
\]
where $q$ goes from $\max\{0,d-t\}$ to $\min\{d-1,d-1-i-t+s,i\}$, then the multiplication $\cdot \ell^t:A_i \to A_{i+t}$ for a general linear form $\ell = \overline{\ell} + x\in A_1$ has full rank if and only if all the maps in $\mathcal{L}$ are
\begin{align*}
&\bullet \text{injective} &\text{ if } i<\min\{d-1,d-1-i-t+s\}, \\
&\bullet \text{surjective} &\text{ if } d-1-i-t+s<\min\{d-1,i\}, \\
&\bullet \text{of full rank for the same reason} &\text{otherwise}.
\end{align*}
\end{remark}

While the several occurrences of maximums and minimums in Theorem \ref{thm:Rank-recursion} can make the statement feel a bit technical, they mostly have intuitive explanations. For example, the maximums and minimums in the product for $F_1$ are exactly corresponding to maps $\cdot \overline{\ell}^{2q + t -(d-1)}:B_{i-q} \to B_{i+q+t-(d-1)}$ from Theorem \ref{thm:Rank-recursion_v1} that do not have full rank for trivial reasons. That is, $q\geq d-t$ ensures that the exponent $2q+t-(d-1)$ is positive, $q\leq i$ that we are mapping from a non-zero space, and $q\leq d-1-i-t+s$ that we are mapping to a non-zero space.  

Notice that if we set $d=2$ in Theorem \ref{thm:Rank-recursion}, unwinding the definitions then gives Theorem \ref{thm:recursion_quadratic} and Proposition \ref{prop:quadratic_edge_case}, so they are indeed special cases of Theorem \ref{thm:Rank-recursion}. 

\section{On monomial complete intersections}\label{sec:mon_CI}

With Theorem \ref{thm:Rank-recursion} at hand, we can now also generalize Proposition \ref{prop:Quadratic_det} to general monomial complete intersections.

\begin{theorem}\label{thm:CI_determinant}
Let $A=k[x_1,\dots, x_n]/ (x_1^{d_1},\dots, x_n^{d_n})$ be a monomial complete intersection over a field of characteristic zero and let $\ell = x_1 + \cdots + x_n$ be a linear form. Next, let $e=\sum_{j=1}^n (d_j-1)$ denote the socle degree of $A$, $\mathbf{d}=(d_1,\dots, d_n)$ and $\mathbf{d}'=(d_1,\dots, d_{n-1})$. Denote by $a_{n,i,\mathbf{d}}$ the absolute value of the determinant of 
\[
\cdot \ell^t:A_i \to A_{e-i}
\]
for $t=e-2i$. Then, for any $i\leq e/2$, we have that 
\[
a_{n,i,\mathbf{d}}=\left(\prod_{j=\max\{0,d_n-t\}}^{\min\{i,d_n-1\}}a_{n-1,i-j,\mathbf{d}'}\right)  \prod_{s=\max\{1,d_n-i-1\}}^{\min\{t-1,d_n-1\}}|\det(C_{t,s,d_n})|^{h_{i+s-(d_n-1)}-h_{i+s-d_n}}
\]
where $h_j=\dim_k(k[x_1,\dots, x_{n-1}]/(x_1^{d_1},\dots, x_{n-1}^{d_{n-1}}))_j$. In particular, $a_{n,i,\mathbf{d}}\neq 0$ and $A$ has the strong Lefschetz property.
\end{theorem}

\begin{proof}
The proof here follows the same logic as that of Proposition~\ref{prop:Quadratic_det}. First we have that $a_{1,i,d_1}=1$ for all $i=0,1,\dots, d_1/2$. So we may assume that $n>1$ and $a_{n-1,i,\mathbf{d}'}\neq 0$ for any $\mathbf{d'}$ and $i\leq e'/2$ where $e'=\sum_{j=1}^{n-1} (d_j-1)$ is the socle degree of $B=k[x_1,\dots, x_{n-1}]/(x_1^{d_1},\dots, x_{n-1}^{d_{n-1}})$. Next, write $A=B\otimes_k k[x_n]/(x_n^{d_n})$. Then we know that $B$ has a symmetric Hilbert series with socle degree $e'$. Thus, for $t=e-2i$, we have that $h_{i-q}=h_{i+q+t-(d_n-1)}$ since $$(i-q)+(i+q+t-(d_n-1)) = 2i + t - (d_n-1) = e'.$$
Hence all the matrices representing the maps 
\[
\cdot \overline{\ell}^{2q + t -(d_n-1)}:B_{i-q} \to B_{i+q+t-(d_n-1)}
\]
for $q$ from $\max\{0,d_n-t\}$ to $\min\{i,d_n-1, d_n-1-i-t+e'\}$ are square. Therefore Theorem \ref{thm:Rank-recursion} can be applied. Here several simplifications can be made to the factors $F_1$, $F_2$ and $F_3$. First, $d_n-1-i-t+e'=e-i-t=i$, so the product in $F_1$ can be simplified to have the upper bound $\min\{d-1,i\}$ for $q$. Next, since $2i-(d_n-1)+t-e'=0$, we can include $F_3$ as a factor in the product for $F_2$ since then $\dim_k(B_{i+(i+t-e')-d_n})=0$ and instead let the product in $F_2$ start at $\max\{1,i+t-e',d_n-i-1 \}$. Finally, $i+t-e'=d_n-i-1$ by the definition of $t$ and $e'$, so we have $\max\{1,i+t-e',d_n-i-1 \}=\max\{1,d_n-1-i\}$. The product of the new $F_1$ and $F_2$ then gives the desired determinant formula. Finally, since all determinants $\det(C_{t,s,d_n})$ are non-zero by Equation \eqref{eq:closed_determinant} and all $a_{n-1,i-j,\mathrm{d}'}$ are non-zero by assumption, it follows that $a_{n,i,\mathbf{d}}$ is also non-zero, giving that $A$ has the strong Lefschetz property.
\end{proof}

Note that $a_{n-1,i-j,\mathbf{d}'}$ is the absolute value of the matrix representing multiplication by $\overline{\ell}^{2j+t-(d_n-1)}$ from degree $i-j$. Thus the power of the mapping corresponding to $a_{n-1,i-j,\mathbf{d}'}$ do also change, even though it is not directly displayed in the notation.

\begin{example}
Let us use Theorem \ref{thm:CI_determinant} to calculate the absolute value of the determinant of 
\[
\cdot (x+y+z)^2:A_2\to A_4
\]
when $A=k[x,y,z]/(x^2, y^3, z^4)$. If we view $z$ as the last variable, then in the notation of Theorem \ref{thm:CI_determinant} we want to calculate $a_{3,2,(2,3,4)}$ and have that $i=t=2$, $d_3=4$, and $h_s=\dim_k(k[x,y]/(x^2,y^3))_s$. The recursion then gives that
\begin{align*}
a_{3,2,(2,3,4)} &= \left(\prod_{j=2}^{2}a_{2,2-j,(2,3)}\right)\left(\prod_{s=1}^1|\det(C_{2,s,4})|^{h_{s-1}-h_{s-2}} \right)\\
&=a_{2,0,(2,3)} \cdot |-4|^1\\ 
&= 4a_{2,0,(2,3)}
\end{align*}
where we used that $h_{-1}=0$. The determinants $|\det(C_{t,p,d})|$ can here be calculated using Krattenthaler's formula from Equation \eqref{eq:closed_determinant} or directly from their definitions since they are small enough in this case. To calculate $a_{2,0,(2,3)}$, we can use the recursion again with $i=0$ and $t=d_2=3$  to get that
\[
a_{2,0,(2,3)} = a_{1,0,(2)}|\det(C_{3,2,3})|^{1} = 3
\]
where we used that $a_{1,0,(2)}=1$. Hence we conclude that our sought determinant equals $$a_{3,2,(2,3,4)}=4\cdot 3 = 12.$$ 
This is not the only way we could calculate $a_{3,2,(2,3,4)}$. If we change the order of the variables, we get another recursion. Indeed, let us view $x$ as the last variable and calculate $a_{3,2,(4,3,2)}$. As a first step, we have $i=t=d_3=2$, and the recursion gives 
\[
a_{3,2,(4,3,2)} = a_{2,2,(4,3)} a_{2,1,(4,3)} |\det(C_{2,1,2})|^{3-2} = 2 a_{2,2,(4,3)} a_{2,1,(4,3)}.
\]
Two more applications of the recursion then gives that
\[
a_{2,2,(4,3)} = a_{1,0,(4)}|\det(C_{1,0,3})|^1 = 1
\]
and 
\[
a_{2,1,(4,3)} = a_{1,0,(4)} a_{1,1,(4)} |\det(C_{3,2,3})|^{1-1} |\det(C_{3,1,3})|^1 = 1 \cdot |-6| = 6.
\]
Hence 
\[
a_{3,2,(4,3,2)}= 2 \cdot 6 = 12 = a_{3,2,(2,3,4)}
\]
as expected.
\end{example}

To our knowledge, Theorem \ref{thm:CI_determinant} gives the first known formula, recursive or not, for the determinants appearing in the study of the SLP of arbitrary monomial complete intersections. Previous results of this flavor includes Hara and Watanabe's recursion in the quadratic case, and when the socle degree of the monomial complete intersection is odd, the WLP of such an algebra is determined by a single determinant and given by Proctor in \cite[Corollary 1]{Proctor} using the language of weak compositions. See also \cite[Theorem 3.1]{Cook_pos_char} for a modern account of that result. 

One drawback of the determinant formula given in Theorem \ref{thm:CI_determinant} is that it gives the result in terms of other determinants of the form $\det(C_{t,p,d})$. If one could understand the prime factors of such determinants, it then seems possible that one would be able to say when monomial complete intersections over fields of positive characteristic have the weak or the strong Lefschetz property, or perhaps even give a version of Theorem \ref{thm:quadratic_criterion} for arbitrary monomial complete intersections. Note that the strong Lefschetz property has already been characterized by Lundqvist and Nicklasson \cite{SLP_pos_char} and Nicklasson \cite{Sum_of_var_SLP}, and the weak Lefschetz property for equigenerated monomial complete intersections by Kustin and Vraciu \cite{Kustin_Vraciu}, but the WLP is in general still open. 

For example, let us apply Theorem \ref{thm:CI_determinant} to $A=k[x,y,z]/(x^{d_1}, y^{d_2}, z^{d_3})$ where we assume that $d_1\leq d_2 \leq d_3$ and $d_1+d_2+d_3=2r$ is even. In this case, $A$ has the WLP if and only if multiplication by a linear form from degree $i=r-2$ has full rank. If $d_3>d_1+d_2-2$, both products for $a_{3,i,(d_1,d_2,d_3)}$ in Theorem \ref{thm:CI_determinant} are empty and $A$ has the WLP. Else, $a_{3,i,(d_1,d_2,d_3)}= a_{2,r-d_3-1,(d_1,d_2)}$. This time, applying Theorem \ref{thm:CI_determinant} gives a first product that is one since we always have $a_{1,p,d}=1$, while the second product is only contributing with $\det\left(C_{d_3, \frac{d_3+d_2-d_1}{2}, d_2}\right)$ because all other factors are of the form $|\det(C_{d_3,s,d_2})|^0$. Hence $A$ has the WLP if and only if $d_3>d_1+d_2-2$ or the characteristic $p$ of $k$ does not divide $\det\left(C_{d_3, \frac{d_3+d_2-d_1}{2}, d_2}\right)$. This recovers half of the main result of Li and Zanello \cite[Theorem 3.2. (1)]{Li_Zanello}. 

However, finding the prime factors of these determinants turns out to be a highly nontrivial task. For instance, if $d_1=d_2=d_3=2m$ in the previous example, then $A=k[x,y,z]/(x^{2m}, y^{2m}, z^{2m})$ has the WLP if and only if $p$ does not divide $\det(C_{2m,m,2m})$ by the above. But by Brenner and Kaid \cite[Theorem 2.6]{Brenner_Kaid_pos_char}, $A$ fails the WLP if and only if there are integers $s$ and $t$ such that
\[
\frac{6m}{6t+2}> p^s > \frac{6m}{6t+4}.
\]
When $p=2$, there is also another characterization of WLP for the general case using a Nim-sum condition \cite[Theorem 8.1]{Nim_sum}. It should be mentioned that the determinants $\det(C_{t,p,d})$ have appeared in the study of the Lefschetz properties before, such as in Altafi and Nemati \cite{Altafi_Nemati}, Kustin and Vraciu \cite{Kustin_Vraciu}, and in Cook II and Nagel \cite{Cook_Nagel_Lozenges_Illinois}. They are especially prevalent in the works of Cook II and Nagel where they use these determinants to give several connections between the WLP and Lozenge tilings. The first combinatorial explanation for some of this was given by a bijective proof of Chen, Guo, Jin and Liu in \cite{Bijective_MacMahon_det}. In particular, if $d_1=a+b$, $d_2=a+c$ and $d_3=b+c$ for some $a,b,c\in \mathbb{N}$, then the determinant $\det(C_{b+c, c, a+c})$ deciding if $A=k[x,y,z]/(x^{a+b}, y^{a+c}, z^{b+c})$ has the WLP or not is exactly counting the number of plane partitions of an $a\times b \times c$ box. Hence, one can use a classical formula of MacMahon \cite{MacMahon} to get that
\[
|\det(C_{t,p,d})| = \frac{\mathcal{H}(t-p)\mathcal{H}(p)\mathcal{H}(d-p)\mathcal{H}(t+d-p)}{\mathcal{H}(t)\mathcal{H}(d)\mathcal{H}(t+d-2p)}
\]
where
\[
\mathcal{H}(n) = \prod_{i=0}^{n-1}i!
\]
is the hyperfactorial of $n$. Comparing this with Krattenthaler's formula from Equation \eqref{eq:closed_determinant} and other expressions such as the one by Roberts \cite[page 335]{Roberts_local_cohomology},
\[
|\det(C_{t,p,d})| = \frac{\binom{t}{p}\binom{t+1}{p}\cdots \binom{d+t-p-1}{p}}{\binom{p}{p}\binom{p+1}{p}\cdots \binom{d-1}{p}},
\] 
we do have a number of different ways to express $|\det(C_{t,p,d})|$, just not anyone that gives an easy way for finding its prime factors. 

\section{Applications on Lefschetz properties}\label{sec:Applications}

The goal for this section is to use Theorem \ref{thm:Rank-recursion} to classify in which cases algebras or families of algebras of the form $A=B\otimes_k k[x]/(x^d)$ do have the weak or the strong Lefschetz property. In particular, Theorem \ref{thm:Extend_SLP} gives a proof of a slight variation of a theorem of Lindsey \cite[Theorem 3.10]{Lindsey}.
But to begin with, a first corollary of Theorem \ref{thm:Rank-recursion} is how the WLP of an extension behaves. 

\begin{corollary}\label{cor:extend_WLP}
Let $A=B\otimes_k k[x]/(x^d)$. Then $A$ has the WLP if and only if there is a linear form $\ell\in B_1$ such that
\[
\cdot \ell^d:B_i \to B_{i+d}
\]
has maximal rank for any $i\geq 0$. 
\end{corollary}

\begin{proof}
First, note that when $t=1$ in Theorem \ref{thm:Rank-recursion}, the product in Equation \eqref{eq:prime_cond} is empty, so the theorem can be used for fields of any characteristic. We then get that $\ell + x$ has full rank on $A$ if and only if
\[
\cdot \ell^{2q + 1 -(d-1)}:B_{i-q} \to B_{i+q+1-(d-1)}
\]
has full rank for any $i$ and $q=d-1$. That is, $A$ has the WLP if and only if
\[
\cdot \ell^d: B_{i-(d-1)} \to B_{i+1}
\]
has full rank for all $i$ as desired.
\end{proof}

In the case when $B$ is a monomial complete intersection, Corollary \ref{cor:extend_WLP} appears for example as \cite[Lemma 4.7]{SLP_pos_char}, and is probably well known but maybe not written down in this generality. 

As a consequence of Corollary \ref{cor:extend_WLP}, one does also get an easy answer for when adding a new variable always creates an algebra with the WLP.

\begin{corollary}\label{cor:extension_has_WLP}
An algebra $B$ has the strong Lefschetz property if and only $B\otimes_k k[x]/(x^d)$ has the weak Lefschetz property for all $d\geq 1$. 
\end{corollary}

\begin{proof}
By Corollary \ref{cor:extend_WLP}, the $t$:th power of a linear form gives a map that has full rank in all degrees on $B$ if and only if $B\otimes_k k[x]/(x^t)$ has the WLP. Hence all powers of such a linear form have full rank, meaning that $B$ has the SLP, if and only if $B\otimes_k k[x]/(x^t)$ has the WLP for all $t\geq 1$.
\end{proof}

Corollary \ref{cor:extension_has_WLP} has also been established before, for example in \cite[Remark 3.1.7]{Lindsey_thesis}. Next, we remark that the addition of a new variable can not make it easier to have the SLP.

\begin{proposition}\label{prop:extension_not_easier}
If $A=B\otimes_k k[x]/(x^d)$ has the SLP for some $d\geq 1$ and $k$ is a field of characteristic zero, then $B$ has the SLP.
\end{proposition}

\begin{proof}
Say we want to show that $$\cdot \ell^a:B_j \to B_{j+a}$$ has full rank for some $a\geq 1$ and $j\geq 0$. By Theorem \ref{thm:Rank-recursion}, we know that if
\[
\cdot \ell^{d-1+a}:A_j \to A_{j+a+d-1}
\]
has full rank then all of the maps
\[
\cdot \ell^{2q+a}:B_{j-q} \to B_{j+a+q}
\]
must have full rank for $q$ from $\max\{0,1-a\}=0$ to $d-1$. Since $d\geq 1$, we can use that $A$ has the SLP to get that the map obtained when $q=0$ has full rank. That is, 
$$\cdot \ell^a:B_j \to B_{j+a}$$ has full rank and hence $B$ has the SLP.
\end{proof}

Even though Proposition \ref{prop:extension_not_easier} may seem intuitive, the same statement is not true if you replace the SLP with the WLP. For example, $A=\mathbb{Q}[x,y,z]/(x^3, y^3, z^3, xyz)$ fails the WLP while $A\otimes_{\mathbb{Q}}\mathbb{Q}[w]/(w^2)$ does have the WLP. 

Looking at Corollary \ref{cor:extension_has_WLP} again, one may wonder if there are any criteria on $B$ that ensures that $B\otimes_k k[x]/(x^d)$ has the SLP for any $d\geq 1$. To answer that question, we need the following definition.

\begin{definition}
Let $A$ be an artinian algebra with Hilbert series given by 
\[
\HS(A;T)=\sum_{i=0}^Dh_iT^i.
\]
Then $A$ is \emph{almost centered} if both of the following two properties are satisfied. 
\begin{itemize}
\item If $h_i<h_j$ for some $i<j$, then $h_{i-s}\leq h_{j+s}$ for any $s \geq 0$.
\item If $h_i>h_j$ for some $i<j$, then $h_{i-s}\geq h_{j+s}$ for any $s\geq 0$.
\end{itemize}
\end{definition}

Since we know from Corollary \ref{cor:extension_has_WLP} that an algebra $B$ must have the SLP for all of the extensions $B\otimes_k k[x]/(x^d)$ to have the WLP, we can now show that if we moreover assume that $B$ is almost centered, then all of these extension do also have the SLP.

\begin{theorem}\label{thm:Extend_SLP}
Let $B$ be a standard graded artinian $k$-algebra where $k$ is a field of characteristic zero. Then $A=B\otimes_k k[x]/(x^d)$ has the SLP for any $d\geq 1$ if and only if $B$ is almost centered and has the SLP.
\end{theorem}

\begin{proof}
Let $\ell$ be a general linear form in $B$ and set $h_i=\dim_k B_i$. Since $k$ is of characteristic zero, we know from Theorem~\ref{thm:Rank-recursion} that if $A$ has the SLP or not is completely determined by the family $\mathcal{L}$ of maps
\[
\cdot \ell^{2q + t -(d-1)}:B_{i-q} \to B_{i+q+t-(d-1)}
\]
for $q$ from $\max\{0,d-t\}$ to $d-1$ for $i,t\geq 0$. 

Assume first that $B$ is almost centered and has the SLP. Then all maps in $\mathcal{L}$ have full rank since $B$ has the SLP. Further, they all have full rank for the same reason. Indeed, let $p$ be the smallest value of $q$ in the range $\max\{0,d-t\}$ to $d-1$ such that $h_{i-q}\neq h_{i+q+t-(d-1)}$, say $h_{i-p}>h_{i+p+t-(d-1)}$. Then the map between these degrees is surjective, and as $h_{i-(p+s)}\geq h_{i+(p+s)+t-(d-1)}$ for all $s\geq 0$ from the definition of almost centered, all the other maps in $\mathcal{L}$ are also surjective. The case when $h_{i-p}<h_{i+p+t-(d-1)}$ is analogous, establishing injectivity. Thus we get that $A$ has the SLP for any $d\geq 1$.

Suppose now that $A$ has the SLP for any $d\geq 1$. Then in particular it has the WLP, so we know from Corollary~\ref{cor:extension_has_WLP} that $B$ has the SLP. We will now prove that $B$ is almost centered by contradiction. Assume for a contradiction that $h_i<h_j$ for some $i<j$ but that $h_{i-s}>h_{j+s}$ for some $s\geq 0$. Consider now $A=B\otimes_k k[x]/(x^d)$ for $d=s+1$ and multiplication by the $t$:th power of a linear form in $A$ from degree $i$ for $t=j-i+d-1$. Then Theorem~\ref{thm:Rank-recursion} gives that all maps in $\mathcal{L}$ should have full rank for the same reason. Here $q=0$ gives one of the maps since
\[
d-t = s+1-(j-i+d-1) = i-j+1 \leq 0,
\]
and $q=d-1=s$ another part. Now, for $q=0$, we are going from a space of dimension $h_i$ to a space of dimension $h_{i+t-(d-1)}=h_j$, so $h_i<h_j$ gives that injectivity is required. On the other hand, for $q=s$, we are going from a space of dimension $h_{i-s}$ to one of dimension $h_{j+s}$, so $h_{i-s}>h_{j+s}$ are forcing surjectivity. Hence they cannot have full rank for the same reason and we have our contradiction. 
The case for the reverse inequalities can be proven in the same way for the same choices of $d$ and $t$. Hence $B$ must be almost centered and we are done.
\end{proof}

Almost centered algebras were first introduced by Lindsey in \cite{Lindsey}, but using another criterion for the Hilbert series of such algebras that we give in the next corollary. However, using Theorem~\ref{thm:Extend_SLP}, we can prove that they are equivalent.  

\begin{corollary}\label{cor:Mid-heavy=a.c.}
An artinian algebra $A$ with the SLP and socle degree $D$ is almost centered if and only if it has a Hilbert series satisfying that either 
\begin{itemize}
\item $h_{i-1}\leq h_{D-i}\leq h_i$ for all $ 0 \leq i \leq \lfloor \frac{D}{2} \rfloor$, or 
\item $h_{D-i+1}\leq h_{i}\leq h_{D-i}$ for all $ 0 \leq i \leq \lfloor \frac{D}{2} \rfloor$.
\end{itemize}
\end{corollary}

\begin{proof}
By \cite[Theorem 3.10]{Lindsey}, an algebra $A$ with the SLP satisfies one of the criteria given in the statement of this corollary if and only if $A\otimes_k k[x]/(x^d)$ has the SLP for all $d\geq 1$. But by Theorem~\ref{thm:Extend_SLP}, that is the case if and only if $A$ has the SLP and is almost centered. Therefore, $A$ must be almost centered if and only if any of the two criteria given in the corollary holds.
\end{proof}

Note that there is a small typo in \cite{Lindsey} where the inequalities are only given for $i\geq 1$ instead of $i\geq 0$ as we have in Corollary \ref{cor:Mid-heavy=a.c.}. This does not add anything to the second criteria, but it does add the requirement $h_D\leq h_0$ to the first criteria. We can see that this is necessary to have since otherwise the algebra $B=\mathbb{Q}[x,y]/(x,y)^3$ has the SLP and satisfy the first criteria, but a short computation shows that $B\otimes_{\mathbb{Q}}\mathbb{Q}[z]/(z^2)$ fails the SLP. However, $B$ fails both criteria when $h_D\leq h_0$ is included. Note further that it is possible to give a direct proof of Corollary \ref{cor:Mid-heavy=a.c.} using the definitions given and without assuming that $A$ has the SLP, but since it does not give any additional insight, we refrain from doing so.

After having explored what criteria that needs to be imposed on an algebra such that algebras in one more variable has the weak or strong Lefschetz property, a reasonable next question might be what is required when we want to add more than one new variable. 

\begin{proposition}
\label{prop:add_moore_variables}
Let $B$ be a standard graded artinian $k$-algebra where $k$ is a field of characteristic zero. Then $$B\otimes_k k[x_1]/(x_1^{d_1})\otimes_k \cdots \otimes_k k[x_n]/(x_n^{d_n}) \cong B\otimes_k  k[x_1,\dots, x_n]/(x_1^{d_1},\dots, x_n^{d_n})$$
has the SLP for any $d_1,\dots, d_n\geq 1$ if and only if $B$ is almost centered and has the SLP. 
\end{proposition}

\begin{proof}
We know from Theorem \ref{thm:Extend_SLP} that if $B\otimes_k  k[x_1,\dots, x_n]/(x_1^{d_1},\dots, x_n^{d_n})$ has the SLP for any $d_1,\dots, d_n\geq 1$, then it must be the case that $B$ is almost centered and has the SLP. Conversely, \cite[Corollary 3.4]{Lindsey} says that if $B$ is almost centered, then so is $B\otimes_k k[x_1]/(x_1^{d_1})$ for any $d_1\geq 1$. Hence repeated applications of Theorem~\ref{thm:Extend_SLP} gives that $B\otimes_k  k[x_1,\dots, x_n]/(x_1^{d_1},\dots, x_n^{d_n})$ has the SLP as desired.
\end{proof}

We note that this also follows from \cite[Theorem 3.5]{Lindsey} since a complete intersection has a symmetric Hilbert function. It might be surprising that it suffices to ask for the WLP and take all $d_i=2$ in Proposition \ref{prop:add_moore_variables} to get the same requirements on the original algebra.

\begin{proposition}
\label{prop:add_squares}
Let $A$ be a standard graded artinian $k$-algebra where $k$ is a field of characteristic zero. Then $$A\otimes_k k[x_1]/(x_1^2)\otimes_k \cdots \otimes_k k[x_n]/(x_n^2) \cong A\otimes_k  k[x_1,\dots, x_n]/(x_1^2,\dots, x_n^2)$$
has the WLP for any $n\geq 0$ if and only if $A$ is almost centered and has the SLP. 
\end{proposition}

\begin{proof}
Let us introduce the notation
\[
A^{[n]}=A\otimes_k  k[x_1,\dots, x_n]/(x_1^2,\dots, x_n^2).
\]
If $A$ is almost centered and has the SLP, then Proposition \ref{prop:add_moore_variables} gives that $A^{[n]}$ has the SLP for all $n\geq 0$ and therefore also the WLP. For the converse direction, assume that $A$ fails the SLP or that $A$ is not almost centered.
If $A$ fails the SLP, then there is some $j> i\geq 0$ such that $\cdot \ell^{j-i}: A_i \to A_{j}$ fails to have full rank. We then claim that $A^{[j-i-1]}$ fails the WLP in degree $j-1$. Indeed, by Theorem~\ref{thm:recursion_quadratic}, $\cdot \ell:A^{[j-i-1]}_{j-1} \to A^{[j-i-1]}_{j}$ having full rank implies that $\cdot \ell^2:A^{[j-i-2]}_{j-2} \to A^{[j-i-2]}_{j}$ has full rank where we let $\ell$ denote the general linear form in the appropriate $A^{[k]}$. Applying Theorem \ref{thm:recursion_quadratic} one more time then gives that $\cdot \ell:A^{[j-i-3]}_{j-2} \to A^{[j-i-3]}_{j-1}$ and $\cdot \ell^3:A^{[j-i-3]}_{j-3} \to A^{[j-i-3]}_{j}$ have full rank for the same reason. After a total of $j-i-1$ applications of Theorem \ref{thm:recursion_quadratic}, we then conclude that the maps
\[
\cdot \ell^{j-i-2p}:A_{i+p} \to A_{j-p}
\]
for $p$ from $0$ to $\lfloor \frac{j-i-1}{2} \rfloor$ all have full rank for the same reason. In particular, $p=0$ gives that $\cdot \ell^{j-i}:A_i \to A_{j}$ should have full rank, a contradiction. Next, assume that $A$ is not almost centered, say $h_i>h_j$ for some $i<j$ but $h_{i-s}<h_{j+s}$ for some $s\geq 0$ where $h_i=\dim_k (A_i)$. The proof in the case of the reverse inequalities is analogous. Let $s'=\min\{s,i\}$ and consider multiplication by $\ell$ on $A^{[2s'+j-i-1]}$ from degree $j+s'-1$. Similar to the above, we then find that the maps
\[
\cdot \ell^{2s'+j-i-2p}: A_{i-s'+p} \to A_{j+s'-p}
\]
for $p$ from $0$ to $\lfloor \frac{2s'+j-i-1}{2} \rfloor$ all have full rank for the same reason. If $s'=s$, then $p=0$ gives a map that can only be injective since $h_{i-s}<h_{j+s}$, while $p=s$ gives a map that can only be surjective since $h_{i}>h_j$, contradicting our assumption that $A^{[2s'+j-i-1]}$ has the WLP. Hence we get that if $h_i<h_j$, then $h_{i-s}\leq h_{j+s}$ for at least all $s$ satisfying $0\leq s \leq i$.

Finally, if $s'=i<s$, let us look at multiplications by $\ell$ on $A^{[i+j+1]}$ from degree $i+j$. Then the maps 
\[
\cdot \ell^{i+j-2p+1}: A^{[1]}_{p} \to A^{[1]}_{i+j-p+1}
\]
for $p$ from $0$ to $\lfloor \frac{i+j}{2} \rfloor$ all have full rank for the same reason. However, here we note that $\dim_k(A^{[1]}_{0})=h_0=1$ and $\dim_k (A^{[1]}_{i+j+1})=h_{i+j}+h_{i+j+1}\geq 2$ where we use that $i+j+1\leq s+j$, so $h_{j+s}>0$ gives that $h_{i+j},h_{i+j+1}\geq 1$. Hence $p=0$ gives a map that can only be injective. But using what we just proved to get from $h_{i}>h_j$ to $h_{i-1}\geq h_{j+1}$, we have that $$\dim_k(A^{[1]}_{i})=h_{i-1}+h_{i} > h_{j} + h_{j+1}=\dim_k(A^{[1]}_{j+1}),$$
so $p=i$ gives a map that can only be surjective. Hence we get a contradiction and that $A$ must be almost centered.
\end{proof}

Combining the results of Theorem \ref{thm:Extend_SLP}, Corollary \ref{cor:Mid-heavy=a.c.}, Proposition \ref{prop:add_moore_variables} and Proposition~\ref{prop:add_squares}, we get the following new characterization of the almost centered property.

\begin{theorem}
Let $A$ be a standard graded artinian $k$-algebra of socle degree $D$ where $k$ is a field of characteristic zero. Then the following are equivalent.
\begin{itemize}
\item $A$ is almost centered and has the SLP.
\item $A$ has the SLP and a Hilbert series satisfying that $h_{i-1}\leq h_{D-i}\leq h_i$ for all $ 0 \leq i \leq \lfloor \frac{D}{2} \rfloor$, or $h_{D-i+1}\leq h_{i}\leq h_{D-i}$ for all $ 0 \leq i \leq \lfloor \frac{D}{2} \rfloor$.
\item $A\otimes_k k[x]/(x^d)$ has the SLP for all $d\geq 1$.
\item $A\otimes_k  k[x_1,\dots, x_n]/(x_1^{d_1},\dots, x_n^{d_n})$ has the SLP for all $d_1,\dots, d_n\geq 1$.
\item $A\otimes_k  k[x_1,\dots, x_n]/(x_1^{2},\dots, x_n^{2})$ has the WLP for all $n\geq 0$.
\end{itemize}
\end{theorem}

\section{Summary and outlook}

To summarize, the main result of this paper is Theorem \ref{thm:Rank-recursion}, giving a classification and powerful tool for when powers of general linear forms have full rank on artinian algebras of the form $A=B\otimes_kk[x]/(x^d)$ in terms of powers of linear forms on $B$. This is then used to give a short proof that monomial complete intersections have the strong Lefschetz property in Theorem \ref{thm:CI_determinant}. Finally, more characterizations for when different families of algebras of the form $A=B\otimes_kk[x]/(x^d)$ have the weak or the strong Lefschetz properties are given in Section \ref{sec:Applications}.

We end by mentioning some open problems that could serve as natural continuations of this work. First, while Theorem \ref{thm:Rank-recursion} does give us all information we require when working over a field of characteristic zero, it remains to give a full classification when working over finite fields. In particular, it is still an open problem to classify monomial complete intersections with the WLP in positive characteristics. Perhaps a finer study of the matrices involved in the proof of Theorem~\ref{thm:Rank-recursion} could give some insight. 

On the other hand, even when working over a field of characteristic zero, one might wish for a classification for when different kinds of extensions of an algebra $B$ by another algebra $C$ have the weak or strong Lefschetz property from information about $B$. Theorem~\ref{thm:Rank-recursion} does this when $C$ a monomial complete intersection, but other families would also be interesting to consider. Maybe $C=k[x_1,\dots, x_n]/(x_1,\dots, x_n)^d$ for some $d\geq 1$ could be a natural next candidate to look at?

\section*{Acknowledgements}
The author would like to thank Samuel Lundqvist and Lisa Nicklasson for comments that improved the exposition of this paper. 

\bibliographystyle{plain}
\bibliography{references}

\end{document}